\documentclass[10pt]{amsart}

\usepackage{amssymb,amsmath,amsthm,amscd}

\advance\oddsidemargin by -1cm
\advance\evensidemargin by -1cm
\newtheorem{theorem}{Theorem}

\newtheorem{proposition}{Proposition}

\newtheorem{lemma}{Lemma}
\newtheorem{definition}{Definition}
\newtheorem{assumption}{Assumption}
\theoremstyle{definition}
\newtheorem{remark}{Remark}

\newcommand{\bdm}{\begin{displaymath}}
\newcommand{\edm}{\end{displaymath}}
\newcommand{\bq}{\begin{equation}}
\newcommand{\eq}{\end{equation}}
\newcommand{\bqn}{\begin{equation*}}
\newcommand{\eqn}{\end{equation*}}

\newcommand{\kmu}{k^{-1}}
\newcommand{\rn}{\mathbb{R}^n}
\newcommand{\rnn}{\mathbb{R}^{2n}}
\newcommand{\eps}{\varepsilon}
\newcommand{\oh}{\frac{1}{2}}
\newcommand{\phr}{\psi_{wk}}

\newcommand{\norm}[1]{\left\| #1 \right\|}
\newcommand{\mklm}[1]{\left\{ #1 \right\}}
\newcommand{\eklm}[1]{\left\langle #1 \right\rangle}

\renewcommand{\d}{\,d}
\newcommand{\N}{{\mathbb N}}

\newcommand{\C}{{\mathbb C}}
\newcommand{\R}{{\mathbb R}}
\newcommand{\A}{{\mathcal A}}

\newcommand{\D}{{\mathcal D}}
\newcommand{\E}{{\mathcal E}}
\newcommand{\F}{{\mathcal F}}

\renewcommand{\H}{{\mathcal H}}

\newcommand{\M}{{\mathcal M}}
\newcommand{\T}{{\rm T}}

\newcommand{\Ocal}{{\mathcal O}}

\newcommand{\1}{{\bf 1}}
\renewcommand{\epsilon}{\varepsilon}
\renewcommand{\phi}{\varphi}
\renewcommand{\rho}{\varrho}

\newcommand{\Cinft}{{\rm C^{\infty}}}
\newcommand{\CT}{{\rm C^{\infty}_c}}

\renewcommand{\L}{{\rm L}}

\newcommand{\Ncal}{{\mathcal N}}

\renewcommand{\S}{{\mathcal S}}

\newcommand{\SO}{\mathrm{SO}}
\renewcommand{\O}{{\mathrm O}}

\newcommand{\g}{{\bf \mathfrak g}}

\renewcommand{\det}{\mathrm{det}\,}

\renewcommand{\Re}{\mathrm{Re}\,}
\newcommand{\vol}{\text{vol}\,}
\newcommand{\dist}{\text{dist}\,}
\newcommand{\Op}{\mathrm{Op}}

\DeclareMathOperator{\supp}{supp}

\DeclareMathOperator{\tr}{tr}
\DeclareMathOperator{\gd}{\partial}

\DeclareMathOperator{\Sing}{Sing}
\DeclareMathOperator{\Reg}{Reg}

\newcommand{\dbar}{{\,\raisebox{-.1ex}{\={}}\!\!\!\!d}}

\begin{document}

\author{Roch Cassanas and Pablo Ramacher}
\title[Reduced Weyl asymptotics for  PDO on bounded domains II]{Reduced Weyl asymptotics for pseudodifferential operators on bounded domains II \\ The compact  group case}
\address{Roch Cassanas and Pablo Ramacher, Georg-August-Universit\"at G\"ottingen, Institut f\"ur Mathematik, Bunsenstr. 3-5, 37073 G\"ottingen, Germany}
\email{ cassanas@uni-math.gwdg.de \\ramacher@uni-math.gwdg.de}
\thanks{This research was financed by the grant RA 1370/2-1 of the German Research Foundation (DFG)}

\begin{abstract} Let $G\subset \O(n)$ be a compact  group of isometries acting on $n$-dimensional Euclidean space $\R^n$, and ${\bf{X}}$ a bounded  domain in $\R^n$ which is transformed into itself under the action of $G$. Consider a  symmetric, classical pseudodifferential operator $A_0$ in $\L^2(\R^n)$ that commutes with the regular representation of $G$, and assume that it is elliptic on $\bf{X}$. We show that the spectrum of  the Friedrichs extension $A$ of the operator $\mathrm{res} \circ A_0 \circ \mathrm{ext}: \CT({\bf{X}}) \rightarrow \L^2({\bf{X}})$ is discrete, and  using the method of the stationary phase, we derive asymptotics for the number  $N_\chi(\lambda)$ of eigenvalues of $A$  equal or less than $\lambda$ and with eigenfunctions in the $\chi$-isotypic component of $\L^2({\bf{X}})$ as $\lambda \to \infty$, giving also an estimate for the remainder term for singular group actions. Since the considered critical set  is a singular variety, we recur to partial desingularization  in order to apply the stationary phase theorem.
\end{abstract}

\maketitle


\section{Introduction}

Let $G\subset \O(n)$ be a compact Lie group of isometries acting on Euclidean space $\R^n$, and ${\bf{X}}$ a bounded open set of $\R^n$  which is transformed into itself  under the action of $G$. Consider the regular representation of $G$
\begin{equation}
\label{eq:T}
\T(k) \phi(x)=\phi(k^{-1}x)
\end{equation}
 in the Hilbert spaces $\L^2(\R^n)$ and $\L^2({\bf{X}})$ of square-integrable functions by unitary operators. As a consequence of the Peter-Weyl theorem, $T$ decomposes into isotypic components according to
\bqn
\L^2(\R^n)=\bigoplus _{\chi \in \hat G} \H_\chi, \qquad \L^2({\bf{X}})= \bigoplus _{\chi \in \hat G} \mathrm{res} \, \H_\chi,
\eqn
where $\hat G$ denotes the set of irreducible characters of $G$,  and $\mathrm{res}:\L^2(\R^n) \rightarrow \L^2({\bf{X}})$ is the natural restriction operator. The spaces $\H_\chi$ are closed subspaces, and the corresponding orthogonal projection operators are given by
\begin{equation}\label{eq:P}
P_\chi={d_\chi} \int\limits_G \overline{ \chi(k)} \T(k) dk,
\end{equation}
where $d_\chi=\chi(\1)$ is the dimension of the irreducible representation belonging to the character $\chi$, and $dk$ denotes the normalized Haar measure on $G$. In what follows, we do not assume that the boundary $\gd {\bf{X}}$ of $\bf X$ is smooth, but only that there exists a constant $c>0$ such that for any sufficiently small  $\rho>0$, $\vol (\gd {\bf{X}})_\rho \leq c \rho$, where  $(\gd {\bf{X}})_\rho=\mklm{x \in \rn: \dist(x, \gd {\bf{X}})< \rho}$, and that $0 \not\in \gd {\bf{X}}$.

 Let now $A_0$ be a  symmetric, classical pseudodifferential operator in $\R^n$ of order $2m$  that commutes with the operators $\mathrm{T}(k)$ for all $k \in G$.  Let $a_{2m}$ be its principal symbol, and assume that there exists a constant $C_0>0$ such that
\bq\label{ellip1}
 a_{2m}(x,\xi)\geq C_0\, |\xi|^{2m},\quad  \forall x\in {\bf X}, \, \forall \xi \in \rn.
\eq
Let $\mathrm{ext}$ denote the natural extension operator by zero. Under condition  \eqref{ellip1}, the operator
\bqn
\mathrm{res} \circ A_0 \circ \mathrm{ext}: \CT({\bf{X}}) \longrightarrow \L^2({\bf{X}}),
\eqn
is symmetric and lower semi-bounded, and we denote its Friedrichs extension by $A$.  It can be shown that  $A$ has  compact resolvent, and if the boundary of $\bf X$ is sufficiently smooth, and $A_0$ satisfies the transmission property,  the domain of $A$ is given by
\begin{equation*}
D(A)=\{u\in H_0^m ({\bf X}) : A_0 u\in \L^2({\bf{X}})\},
\end{equation*}
where $H_0^m ({\bf X})$ is the closure of $\CT({\bf X})$ in the Sobolev space $H^m ({\bf X})$, so that  we are in presence of  a generalized Dirichlet problem. Since  $A$ leaves  each of the isotypic components $\mathrm{res} \, \H_\chi$ invariant, the restriction of $A$ to $\mathrm{res} \, \H_\chi$ gives rise to the so-called reduced operator $A_\chi$. Its domain is $D(A_\chi)=D(A)\cap\mathrm{res} \, \H_\chi $, and its spectrum is discrete,  the spectrum of $A$  being equal to the union of the spectra of the operators $A_\chi$. 

 The purpose of this paper is to investigate the spectral counting function $N_\chi(\lambda)$ of $A_\chi$, which is given by  the number of eigenvalues of $A_\chi$,  counting multiplicities,  that are less than $\lambda\in\R$.
  It corresponds to the number of eigenvalues of $A$ less than $\lambda$,   and with eigenfunctions in the $\chi$-isotypic component of $\L^2({\bf{X}})$, so that
  \bqn
N_\chi(\lambda) =d_\chi \sum_{t \leq \lambda} \mu_\chi(t),
\eqn
where  $\mu_\chi(t)$ denotes the multiplicity of the irreducible representation of dimension $d_\chi$ corresponding to the character $\chi$ in the eigenspace of $A$ with eigenvalue $t$. $N_\chi(\lambda)$ describes the distribution of eigenvalues of $A$, and we shall investigate its asymptotic behavior as $\lambda \to +\infty$  by means of the generalized theorem of the stationary phase. 
It will turn out that $N_\chi(\lambda)$ is intimately related to the representation theory of $G$, and  the geometry of the Hamiltonian action of $G$ on the symplectic manifold $T^\ast({\bf{X}})$. In fact, if $(A_1,\dots, A_d)$ is a basis of the Lie algebra $\g$ of $G$,  let
\bqn
\mathbb{J}:T^\ast({\bf{X}})\simeq {\bf X}\times \rn\to \g\simeq \R^d, \quad (x,\xi) \to (\eklm{A_1x,\xi},\dots,\eklm{A_dx,\xi}),
\eqn
 be the associated momentum map,  where $\eklm{\cdot,\cdot}$ stands  for the Euclidean scalar product in $\rn$, and denote by
 \bqn
 \Omega_0/G=\mathbb{J}^{-1}(\{ 0 \})/G
 \eqn
the symplectic quotient of $T^\ast({\bf{X}})$ at level zero.  This quotient is naturally related to the critical set of the phase function  in question, and plays a crucial role in our reduction. Indeed, we shall prove that   $N_\chi(\lambda)$ is asymptotically determined by  a certain volume of the quotient $\Omega_0/G$, which is  symplectically diffeomorphic to $T^*({\bf X}/G)$ on its smooth part \cite{emmrich-roemer}. Now, the major difficulty in applying the generalized stationary phase theorem in our setting  stems from the fact that, due to the singular orbit structure of the underlying group action,  the zero level $\Omega_0$ of the momentum map, and, consequently, the considered critical set, are in general singular varieties. In fact,  if the $G$-action on $T^\ast({\bf X)}$ is not free, the considered momentum map is no longer a submersion, so that $\Omega_0$ and $\Omega_0/G$ are not smooth anymore. Nevertheless, it can be shown that these spaces have a Whitney stratification into smooth submanifolds, see \cite{ortega-ratiu}, Theorems 8.3.1 and 8.3.2, which corresponds to the stratification of $T^\ast({\bf{X}})$, and $\rn$ into orbit types. To apply the principle of the stationary phase to our problem, we shall therefore proceed to partially resolve the singularities of $\Omega_0$, and then apply the stationary phase theorem in the resolution space under the sole assumption that  the set $\Sing \Omega_0$ of points where  $\Omega_0$ is not a manifold is contained in a strict vector subspace of $T^\ast({\bf{X}})$. This is always fulfilled for group actions that satisfy the following condition \footnote{Examples for such group actions are given in Remark 1.}: If $\rn_{(H_0)}$ denotes the union of all principal orbits in $\rn$ of type $G/H_0$, which is an open and dense subset in $\rn$, then $\rn \setminus\rn_{(H_0)}$ should be contained in a strict vector subspace of $\rn$. The main result  of this paper is Theorem \ref{thm:1}, which states that, as  $\lambda\to +\infty$, one has the asymptotic  formula
\begin{equation*}
N_\chi(\lambda)=\frac {d_\chi [\rho_{\chi|H_0}:1]}{(2\pi)^{n-\kappa}}  \vol([a_{2m}^{-1}((-\infty, 1])\cap \Omega_0]/G) \,    \lambda^{(n-\kappa)/2m} +O(\lambda^{(n-\kappa-1/4)/2m}),
\end{equation*}
where $d_\chi=\chi(\1)$,           $[\rho_{\chi|H_0}:1]$ is the multiplicity of the trivial representation in the restriction of $\rho_\chi$ to any principal isotropy group conjugated to $H_0$, and $\kappa$ the common dimension of the orbits of principal type. The volume  of the quotient $ [a_{2m}^{-1}((-\infty, 1])\cap \Omega_0]/G$ is defined in Section 5.

The asymptotic distribution of eigenvalues was first studied by Weyl \cite{weyl} for certain second order differential operators in $\R^n$ using variational techniques. H\"ormander \cite{hoermander68} then extended these results to elliptic pseudodifferential  operators on closed manifolds using the theory of Fourier integral operators.  The first ones to study reduced Weyl asymptotics for elliptic operators on closed Riemannian manifolds  in the presence of a compact group of isometries were Donnelly \cite{donnelly} together with  Br\"uning and Heintze \cite{bruening-heintze}. In the semi-classical context, reduced Weyl asymptotics and trace formulae were investigated in \cite{helffer-elhouakmi},  and in \cite{cassanas} via coherent states.
Our approach is based on the method of approximate spectral projections, first introduced by Tulovskii and Shubin \cite{tulovsky-shubin}.  Nevertheless, due to the presence of the boundary, the original method cannot be applied to our situation, and one has to use more elaborate techniques, which were subsequently developed by Feigin \cite{feigin} and Levendorskii \cite{levendorskii}. Compared to the method of Fourier integral operators, this approach gives weaker estimates for the remainder, but allows to consider non-smooth boundaries. Recently, Bronstein and Ivrii have obtained even sharp estimates for the remainder term in the case of differential operators on manifolds with boundaries satisfying the conditions specified above \cite{ivrii-bronstein03,ivrii03}.

This paper is the second part  of an investigation initiated in  \cite{ramacher07},  which we shall refer to in the following as Part I. There, the foundations of the calculus of approximate spectral projection operators were provided, and the case of a finite group of isometries was settled.  In this second part, the case of a compact group of isometries will be considered. Before we start, some comments on the  results obtained might be in place. Asymptotics for the spectral counting function $N_\chi(\lambda)$ were obtained in \cite{donnelly} and \cite{bruening-heintze} for general compact, isometric and effective Lie group actions using Heat kernel methods;  nevertheless, this approach  does not allow to derive estimates for the remainder term.
 Using Fourier integral operator techniques, the same authors obtained rather optimal remainder estimates for compact G-manifolds   in the cases where there is only one orbit type, or all orbits have the same dimension. For orthogonal actions in $\rn$, estimates for the remainder where obtained in \cite{helffer-robert86,helffer-elhouakmi} in  case that  the union $\rn_{(H_0)}$ of all principal orbits  is given by $\rn-\mklm{0}$.
In this paper, remainder estimates are obtained in the case that singular orbits are present by partially resolving the singularities of the zero level of the momentum map $\Omega_0$.

\section{Reduced spectral asymptotics and the approximate spectral projection operators}

In this section, we shall review some basic facts in  the theory of  pseudodifferential operators that will be needed in the sequel, and introduce the method of approximate spectral projection operators. For a more detailed exposition, the reader is referred to Part I, Sections 2 and 3. Let $A_0$ be a  classical  pseudodifferential operator of order $2m$ in $\R^n$, regarded as an operator in $\L^2(\R^n)$ with domain $\CT(\R^n)$. In other words, $A_0$ can be represented by an oscillatory integral of the form
 \bqn
 A_0 u (x) =\int \int e^{i(x-y) \xi} a(x,\xi) u(y) dy \dbar \xi,
 \eqn
where its symbol $a(x,\xi)$ has an asymptotic expansion of the form
 \bqn
 a(x,\xi)\sim\sum_{j\geq 0}a_{2m-j}(x, \xi)\, (1-\chi(\xi)),
 \eqn
$\chi$ being a  compactly supported function equal to $1$ in a neighborhood of zero, and the functions $a_{2m-j}$ are  homogeneous of degree $2m-j$ in variable $\xi$.  $a_{2m}$ is called the principal symbol of $A_0$. If $0\leq \rho,\delta\leq 1$, and $\bf Y$ is an open set in $\R^n$, let us denote by
$S^m_{\rho,\delta}({\bf Y}\times\rn)$ the set of smooth functions $\sigma(x,\xi)$ on ${\bf Y}\times \rn$ such that for all compact sets $K$ in ${\bf Y}$, and all multi-indices $\alpha,\beta$, there exist constants $C_{K,\alpha,\beta}>0$ such that
$$|\partial_\xi^\alpha\partial_x^\beta \sigma(x,\xi)|\leq C_{K,\alpha,\beta} \eklm{\xi}^{m-\rho |\alpha|+\delta|\beta|}.$$
Let $L^m_{\rho,\delta}({\bf Y})$ be the class of pseudodifferential operators with symbols in $S^m_{\rho,\delta}({\bf Y}\times\rn)$.
Then, as a local pseudodifferential operator, $A_0 \in \L^{2m}_{1,0}(\R^n)$, see \cite{shubin}, Section 3.7. In what follows, we shall also need certain global spaces of symbols and pseudodifferential operators, which also take decay properties in $x$ into account. They were introduced by H\"{o}rmander within the framework of Weyl calculus of pseudodifferential  operators. Thus,  consider  on $\R^{2n}$ the metric
  \bqn
  \tilde g_{x,\xi}(y,\eta)=(1+|x|^2+|\xi|^2)^\delta |y|^2 + (1+|x|^2 + |\xi|^2 )^{-\rho} |\eta|^2,
\eqn
 where $1 \geq \rho >\delta \geq 0$, and put $h(x,\xi) = (1+|x|^2 +|\xi|^2 )^{-1/2}$.
   \begin{definition} Let $p$ be a $\tilde g$-continuous function.
  The class $\Gamma_{\rho,\delta}(p, \R^{2n})$, $0\leq \delta < \rho \leq 1$, consists of all functions $u \in \Cinft(\R^{2n})$ which for all multiindices $\alpha,\beta$ satisfy the estimates
 \bqn
 |\gd ^\alpha _\xi \gd^\beta_x u(x,\xi)| \leq C_{\alpha \beta}\,  p(x,\xi) \, (1+|x|^2+|\xi|^2)^{(-\rho|\alpha|+\delta|\beta|)/2}.
  \eqn
    In particular, we shall write  $\Gamma^l_{\rho,\delta}(\R^{2n})$ for $\Gamma_{\rho,\delta}(h^{-l},\R^{2n})$, where  $l \in \R$.
  \end{definition}
The class $\Gamma_{\rho,\delta}(p,\R^{2n})$ is also denoted by $S(\tilde g,p)$,  see Part I, Definitions 1 and 3.  Let now $a \in \Gamma_{\rho,\delta}(p,\R^{2n})$,  $0 \leq 1-\rho \leq \delta< \rho \leq 1$, and $\tau \in \R$. Then
  \bqn
Au(x) = \int \int e^{i(x-y)\xi} a ((1-\tau)x+\tau y, \xi  )u(y) dy \, \dbar \xi
  \eqn
  defines a continuous operator in $\S(\R^n)$, respectively $\S'(\R^n)$, see Part I, Corollary 1.  In this case, $a$ is called the $\tau$-symbol of $A$, and the operator $A$ is denoted by  $\Op^\tau(a)$.   If $\tau=1/2$, $a$ is called they Weyl symbol of $A$, and one also writes $\Op^w(a)$ for $A$.  Pseudodifferential operators with real  Weyl symbols give rise to self-adjoint operators. For $\tau=0$ and $\tau=1$ one simply obtains the usual left and right symbols, respectively. 
Our symbol classes will be mainly of  the form   $S(h^{-2\delta}g, p)=\Gamma_{1-\delta,\delta}(p , \R^{2n})$ with
\bqn
 g_{x,\xi}(y,\eta)=|y|^2+ h(x,\xi)^2|\eta|^2,
\eqn
where $p$ is a smooth, positive, g-continuous function, and $0 \leq \delta < 1/2$.
In what follows,  $\Pi_{\rho,\delta}(p,\rn)$ and $\Pi_{\rho,\delta}^l(\rn)$ will denote the classes of pseudodifferential operators with symbols in $\Gamma_{\rho,\delta}(p,\rnn)$, and  $\Gamma_{\rho,\delta}^l(\rnn)$, respectively.  

 
 Consider now a  bounded domain ${\bf X}$  in $\R^n$ with not necessarily smooth boundary $\gd {\bf X}$, and let $a$ be the left symbol of the classical pseudodifferential operator $A_0$. Clearly, $a \in S(g,h^{-2m}, Z \times \R^n)$ for any compact set $Z\subset \R^n$. By changing $a$ outside ${\bf{X}} \times \R^n$, we can therefore assume that $a  \in S(g,h^{-2m})$, so that
\bqn
A_0\in \Pi_{1,0}^{2m}(\rn).
\eqn
Assume now that $A_0$ satisfies the ellipticity condition \eqref{ellip1}.
\begin{lemma}
The ellipticity condition \eqref{ellip1} is equivalent to the existence of constants  $C,M>0$ such that
\bq\label{ellip2}
((A_0+M\1 )u,u)_{L^2({\bf X})}\geq C\,\norm{u}_{H^m({\bf X})}^2,\qquad \forall u\in \CT({\bf X}).
\eq
where $\norm{.}_{H^m({\bf X})}$ is the norm in the Sobolev space $H^m({\bf X})$.
\end{lemma}
\begin{proof}
Since $A_0+M\1$ is a classical symmetric pseudodifferential operator with principal symbol $a_{2m}$, the implication \eqref{ellip2} $\Rightarrow$ \eqref{ellip1} follows with \cite{levendorskii}, Lemma 13.1. Now, let us assume that (1) is fulfilled.
By compactness, there exists a constant $\eps>0$ such that, if ${\bf X}_\eps=\{x\in\rn : \mbox{dist}(x,{\bf X})<\eps\}$, one has
\bq\label{ellip3}
a_{2m}(x,\xi)\geq \frac{C_0}{2}\, |\xi|^{2m},\quad  \forall x\in {\bf X_\eps}, \, \forall \xi \in \rn.
\eq
The restriction of $A_0$ to ${\bf X}_\eps$ is of course in $L^{2m}_{1,0}({\bf X}_\eps)$ since ${\bf X}$ is bounded, and is elliptic in view of (\ref{ellip3}).
It is not properly supported in general but, according to \cite{shubin}, Proposition 3.3, there exists an operator $R$  with smooth kernel $K_R \in \Cinft({\bf X}_{\eps}\times{\bf X}_{\eps})$, and an operator $A_1$ in $L^{2m}_{1,0}({\bf X}_\eps)$ which is properly supported in ${\bf X}_\eps$ such that, on $\L^2({\bf X}_\eps)$,
$$A_0=A_1+R.$$
 $A_1$ is a classical pseudodifferential operator in ${\bf X}_\eps$, with the same principal symbol as $A_0$, and is elliptic on ${\bf X}_\eps$ in view of (\ref{ellip3}). Applying now  the G{\aa}rding inequality as stated in \cite{grigis-sjoestrand}, page 51, one  deduces the existence of a constant $C_1>0$ such that, for all $u\in\CT({\bf X}_\eps)$ with support in $\overline{{\bf X}}$,
$$\Re ((A_1+C_1\1)u,u)_{\L^2({\bf X}_\eps)}\geq \frac{1}{C_1} \norm{u}^2_{H^m({\bf X}_\eps)}.$$
Now, by the  Schwartz inequality, 
\bqn
\norm {Ru}_{\L^2({\bf{X}})}^2=\int_{\bf{X}} |Ru(x)|^2 dx\leq \int_{\bf{X}} \Big (\int_{\bf{X}} |K_R(x,y)|^2 dy \int _{\bf{X}} |u(z)|^2 dz\Big )  dx, \qquad u \in \CT({\bf{X}}),
\eqn
which implies that the restriction of $R$ to $\L^2({\bf X})$ is a bounded operator.
Consequently, there exists a constant $C_2>0$ such that for  $u \in \CT(\bf{X})$
\bqn
(A_0 +C_1\1 )u,u)_{\L^2(\bf{X})} \geq \frac{1}{C_1} \norm{u}^2_{H^m({\bf X}_\eps)} + \Re(Ru,u)_{\L^2(\bf{X})} \geq \frac{1}{C_1} \norm{u}^2_{H^m({\bf X})} -C_2 \norm{u}^2_{\L^2({\bf{X}})},
\eqn
and the assertion of the lemma follows.
\end{proof}
Next note that if $A_0$ were properly supported, then $A_0 \circ
\mathrm{ext}: \CT({\bf{X}}) \rightarrow \CT({\bf{X}}_1)$, where
${\bf{X}}_1$ is some compact set in $\rn$, see \cite{shubin},
Proposition 3.4. By continuity, this map  would extend to a  map
from $\D'({\bf{X}})$ to $\D'({\bf{X}}_1)$, but in general it is not immediately clear if the
restriction of $A_0$ to ${\bf X}$  extends to
$\mathcal{D}'({\bf X})$. Nevertheless, as a pseudodifferential
operator in the class $\Pi^{2m}_{1,0}(\rn)$, the operator
$A_0:\mathcal{S}(\rn)\to \mathcal{S}(\rn) $ extends to a mapping
from $\mathcal{S}'(\rn)$ to $\mathcal{S}'(\rn)$, see
\cite{hoermander79}. Therefore, if $u\in \L^2({\bf X})$, then $(A_0\circ
\mathrm{ext})(u)\in \mathcal{S}'(\rn)$, and via the inclusion
$\mathcal{S}'(\rn)\hookrightarrow \mathcal{D}'({\bf X})$, the
operator $\mathrm{res}\circ A_0\circ \mathrm{ext}$ extends naturally
to an operator from $\L^2({\bf X})$ to $\mathcal{D}'({\bf X})$. Let
us now assume that $A_0$ is symmetric, and that \eqref{ellip1} is
satisfied. Under these circumstances, the previous lemma implies that the  operator \bqn \mathrm{res}
\circ A_0 \circ \mathrm{ext}: \CT({\bf{X}}) \longrightarrow
\L^2({\bf{X}}) \eqn is lower semi-bounded, and we shall denote its
Friedrichs extension by $A$. It is a self-adjoint operator in
$\L^2({\bf{X}})$, and is itself lower semi-bounded. Its spectrum is
real. The following proposition shows that $A$ has compact
resolvent, which implies that the spectrum of $A$ is discrete, i.e.\ it consists 
of a sequence of isolated eigenvalues of finite multiplicity tending
to infinity, while the essential spectrum of $A$ is empty.
\begin{proposition}\label{cpctres} 
As an operator in $\L^2({\bf X})$, $A$ has compact resolvent . Moreover, $D(A)\subset H_0^m({\bf X})$ and 
\bq\label{ellip4}
((A+M)u,u)_{L^2({\bf X})}\geq C\,\norm{u}_{H^m({\bf X})}^2 \qquad \forall u\in D(A).
\eq
\end{proposition}
Here,  $H_0^m ({\bf X})$ denotes the  closure of $\CT({\bf X})$ in $H^m ({\bf X})=\mklm{u \in \D'({\bf X}): \partial^\alpha u\in \L^2({\bf X}), \, |\alpha| \leq m}$ with respect to the Sobolev norm.
\begin{proof}
Put $\tilde{A}=\mathrm{res}\circ A_0\circ \mathrm{ext}:\textrm{C}^\infty_c({\bf X})\to \L^2({\bf{X}})$. In view of (\ref{ellip2}), $\tilde{A}$ is semi-bounded on $\CT({\bf X})$. Let $Q(A)$ be its form domain, that is, the completion of $\CT({\bf X})$ with respect to the norm $p(v)=\sqrt{(\tilde{A}+M)v,v)}$,  see \cite{reed-simonII}, page 177. $Q(A)$ is endowed with the limit norm $\norm{.}_{Q(A)}$ of $p$.
According to (\ref{ellip2}), $Q(A)\subset H^m_0({\bf X})$. Since $A$ is the Friedrichs extension of $\tilde{A}$, one has $D(A)\subset Q(A)$, and we obtain equation (\ref{ellip4}).
Let now  $\lambda<-M$. If $u\in D(A)$, the Schwartz inequality yields 
$$\norm{u}_{H^m({\bf X})}\leq C\norm{(A-\lambda)u}_{\L^2({\bf X})}$$
for some constant $C>0$. Thus, if $v\in \L^2({\bf X})$, 
$\norm{(A-\lambda)^{-1}v}_{H^m({\bf X})}\leq C \norm{v}_{\L^2({\bf X})}$.
Therefore $(A-\lambda)^{-1}$ is a continuous map from $\L^2({\bf X})$ to $H^m_0({\bf X})$. But the injection $H^m_0({\bf X})\hookrightarrow \L^2({\bf X})$ is compact by the Rellich theorem. Consequently, $A$ must have   compact resolvent.
\end{proof}

Consider  now a compact group of isometries  $G\subset \O(n)$ acting on Euclidean space $\R^n$, and assume that the bounded domain $\bf X$  in $\R^n$  is invariant under $G$. Then its boundary is $G$-invariant, too.
Let  $T$ be the unitary representation of $G$ in the Hilbert spaces $\L^2(\rn)$ and $\L^2(\bf{X})$ defined in \eqref{eq:T}, and assume that the operator $A_0$ commutes with the representation $\mathrm{T}$.  The $G$-action on ${\bf{X}}$ induces a Hamiltonian action of $G$ in the cotangent bundle  $T^\ast({\bf{X}})$ of $\bf{X}$ given by   
$$  G\times T^\ast({\bf{X}}) \rightarrow T^\ast({\bf{X}}): (k,x,\xi) \to \sigma_k(x,\xi)=(\kappa_k(x), ^t \kappa _k '(x)^{-1}( \xi))=(\kappa_k(x),\kappa_k(\xi)),$$ 
where we wrote $\kappa_k(x)=kx$. Now, since
\bqn
T(k) \Op^\tau(a) T(k^{-1}) =\Op^\tau( a \circ \sigma_k), \qquad a \in S(\tilde g, p),
\eqn
the $G$-invariance of $A_0$  is equivalent to the $G$-invariance of its symbol, by the uniqueness of the $\tau$-symbol. In particular, the principal symbol $a_{2m}$ of $A_0$ is invariant under $\sigma_k$ for all $k\in G$. Since  the operator $A$ is  also $G$-invariant, the  eigenspaces of $A$ are unitary $G$-modules that  decompose into irreducible subspaces.  The restriction of $A$ to the isotypic component $\mathrm{res} \, \H_\chi$ in the Peter-Weyl decomposition of $(\mathrm{T},\L^2({\bf{X}}))$ is called the reduced operator, and is denoted by $A_\chi$. Its domain is $D(A_\chi)=D(A)\cap\mathrm{res} \, \H_\chi $. As explained in \cite{cassanas},  $A_\chi$ inherits from $A$ the property of having compact resolvent, and the spectrum of $A$ is equal to the union over $\chi$ in $\hat{G}$ of the spectra of the operators $A_\chi$.

 Our purpose in this paper is to investigate the spectral counting function $N_\chi(\lambda)$ of $A_\chi$, which is given by  the number of eigenvalues of $A_\chi$,  counting multiplicities,  that are equal or  less than $\lambda\in\R$.
  It corresponds exactly the number of eigenvalues of $A$ equal or less than $\lambda$,   whose eigenfunctions belong to the $\chi$-isotypic component of $\L^2({\bf{X}})$, so that
  \bqn
N_\chi(\lambda) =d_\chi \sum_{t \leq \lambda} \mu_\chi(t),
\eqn
where  $\mu_\chi(t)$ denotes the multiplicity of the irreducible representation of dimension $d_\chi$ corresponding to the character $\chi$ in the eigenspace of $A$ with eigenvalue $t$.
We shall study $N_\chi(\lambda)$ using the method of approximate spectral projection operators, which was first introduced by Shubin and Tulovskii,  and adapted to the case of bounded domains by Levendorskii.
   It departs from the observation that
$$N(\lambda)=\tr( E_\lambda), $$
where $E_\lambda=\1_{(-\infty,\lambda]}(A)$ is the spectral projector of $A$ belonging to the value $\lambda$. The idea is then to approximate the $E_\lambda$  by means of certain pseudodifferential operators $\E_\lambda$. The trace of $ \E_\lambda$ should then give a good approximation of $N(\lambda)$. The approximate spectral projection operators $\E_\lambda$ will be constructed using Weyl quantization. In order to define them, we introduce now the relevant symbols. Thus, let $a_\lambda  \in S(g,1)$, and $d \in S(g,d)$ be $G$-invariant symbols which, on ${\bf{X}}_\rho \times \mklm{\xi: |\xi| >1}$, ${\bf{X}}_\rho = \mklm{ x: \dist(x, {\bf{X}}) <\rho}$, are given by
 \begin{align*}
  a_\lambda (x,\xi)&=\frac 1 {1+\lambda |\xi|^{-2m}} \Big ( 1 -\frac \lambda {a_{2m}(x,\xi)} \Big ),\\    d(x,\xi)&=|\xi|^{-1},
 \end{align*}
 where $\rho >0$ is some fixed constant, and in addition  assume that $d$ is positive and that $d(x,\xi)\to 0$ as $|x|\to \infty$. We also define
 $$b_\lambda (x,\xi) = a_\lambda(x,\lambda^{1/2m}\xi).$$
We  need to define smooth approximations to the Heaviside function, and to certain characteristic functions on  ${\bf{X}}$.
Thus, let $\tilde \chi$ be a smooth function on the real line satisfying
$0 \leq \tilde \chi \leq 1$, and
\bqn
\tilde \chi(s)=\left \{ \begin{array}{c} 1 \quad \text{for } s<0, \\ 0 \quad \text{for } s>1.
\end{array}\right.
\eqn
  Let $C_0>0$ and $\delta \in (1/4,1/2)$ be constants, and put $\omega=1/2 -\delta$. We then define the $G$-invariant functions
\bqn
\chi_\lambda=\tilde \chi\circ ((a_\lambda+4 h^{\delta-\omega}+8 C_0 d)\,  h^{-\delta}), \qquad \chi^+_\lambda=\tilde \chi\circ ((a_\lambda-4 h^{\delta-\omega}-8 C_0 d)\,  h^{-\delta}),
\eqn
where $0 < \delta -\omega< 1/2$. One can then show that
 $\chi_\lambda,\chi^+_\lambda \in S(h^{-2\delta}g,1)=\Gamma^0_{1-\delta,\delta}(\R^{2n})$ uniformly in $\lambda$, see Part I, Lemma 10. Next, let $U$ be a subset in ${\R^{2n}}$, $c >0$, and put
\bqn
U(c,g)= \mklm{(x,\xi)\in \R^{2n}: \exists   (y,\eta) \in U: g_{(x,\xi)} (x-y,\xi-\eta) < c};
\eqn
according to \cite{levendorskii}, Corollary 1.2, there exists a smoothened characteristic  function $\psi_c\in S(g,1)$  belonging to the set $U$ and the parameter $c$, such that  $\supp \psi_c \subset U(2c,g)$, and ${\psi_c}_{|U(c,g)}=1$.
 Let now
 \bqn
 \M_\lambda=\mklm{(x,\xi)\in \R^{2n}: a_\lambda< 4 h^{\delta-\omega} + 8 C_0 d}.
 \eqn
 Both $\M_\lambda$ and $\gd{\bf{X}} \times \R^n$ are invariant under $\sigma_k$ for all $k \in G$, as well as  $(\gd {\bf{X}}  \times \R^n)(c,h^{-2\delta} g)$, and $\M_\lambda(c,h^{-2\delta}g)$, due to the invariance of $a_{2m}(x,\xi)$, and the considered metrics and symbols. Now, let $\tilde \eta_c, \psi_{\lambda,c}\in S(h^{-2\delta} g,1)$ be smoothened characteristic functions corresponding to the parameter $c$, and the sets $\gd {\bf{X}}  \times \R^n$ and $\M_\lambda$, respectively. According to  Lemma  5 in Part I, we can assume that they are invariant under $\sigma_k$ for all $k \in G$; otherwise consider $ \int _{G} \tilde \eta_c \circ \,\sigma_k \,dk, \, \int_{ G} \psi_{\lambda,c} \circ \,\sigma_k \,dk $, respectively.
 We then define the functions
 \begin{align*}
\eta_{\lambda,-c}(x,\xi)&= \left \{ \begin{array}{cc} 0, & x \notin {\bf{X}}, \\ ( 1-\tilde \eta_c(x,\xi)) \psi_{\lambda,1/c}(x,\xi), & x \in {\bf{X}}, \end{array} \right.\\
\eta_{c}(x,\xi)&= \left \{ \begin{array}{cc} \tilde \eta_c(x,\xi),  & x \notin {\bf{X}}, \\  1, & x \in {\bf{X}}. \end{array} \right.
\end{align*}
Only  the support of $\psi_{\lambda,c}$ depends on $\lambda$, but not its growth properties, so that  $\eta_c,\eta_{\lambda,-c} \in S(h^{-2\delta}g,1)$ uniformly in $\lambda$.
Furthermore, since $\tilde \eta_{2c}=1$ on $\supp \tilde \eta_c$, and $\psi_{\lambda,1/c}=1 $ on $\supp \psi _{\lambda,1/2c}$, on has $\eta_{\lambda,-c}=1$ on $\supp \eta_{\lambda,-2c}$, which implies $\eta_{\lambda,-2c}\eta_{\lambda,-c} = \eta_{\lambda,-2c}.$ Similarly,  one verifies $\eta_c \eta_{2c}=\eta_c$.
We are now ready to define the  approximate spectral projection operators.

  \begin{definition}
The approximate spectral projection operators of the first and second kind are defined by the equations
\bqn
\tilde \E_\lambda= \Op^w (\eta_{\lambda,-2})\, \Op^w(\chi_\lambda)\, \Op^w( \eta_{\lambda,-2}), \qquad \E_\lambda=\tilde \E_\lambda^2(3 -2\tilde \E_\lambda),
 \eqn
while the ones of the third and fourth kind are given by
\bqn
\tilde \F_\lambda= \Op^w (\eta^2_{2}\chi^+_\lambda), \qquad \F_\lambda=\tilde \F_\lambda^2(3 -2\tilde \F_\lambda),
\eqn
\end{definition}
Both $\E_\lambda$ and $\F_\lambda$ are integral operators with kernels in  $\S(\R^{2n})$.
 By Lemma 7.2 in \cite{hoermander79}, this implies that $\E_\lambda$ and $\F_\lambda$ are of trace class and, in particular, compact operators in $\L^2(\R^n)$.  In addition, by  Theorem 2, and the asymptotic expansion {(10)} in Part I,  one has $\sigma ^\tau(\E_\lambda)$, $\sigma^\tau(\F_\lambda)\in S(h^{-2\delta} g,1)$ uniformly in $\lambda$.
 On the other hand, all the involved symbols are real valued, which by general Weyl calculus implies that $ \Op^w (\eta_{\lambda,-2})$, $\Op^w(\chi_\lambda) $, $\Op^w(\eta^2_2 \chi^+_\lambda)$, and consequently also  $\E_\lambda$, and $\F_\lambda$,  are self-adjoint operators in  $\L^2(\R^n)$. Let $P_\chi$ denote the orthogonal projector defined in \eqref{eq:P} onto the isotypic component of the Peter-Weyl decomposition of $(\mathrm{T},\L^2(\rn))$  corresponding to the character $\chi$. By construction, both $\E_\lambda$   and $\F_\lambda$ commute with the projection $P_\chi$,  so that  $P_\chi \E_\lambda$ and $P_\chi \F_\lambda$ are self-adjoint operator of trace class as well.
 Although the decay properties of $\sigma^\tau(\E_\lambda)$,  $\sigma^\tau(\F_\lambda)$  are independent of $\lambda$, their supports do depend on $\lambda$, which will result in  estimates for the trace of $P_\chi\E_\lambda$  and $P_\chi \F_\lambda$ in terms of  $\lambda$ that will be used in order to prove Theorem \ref{thm:1}. In particular, the estimate for the remainder term in Theorem \ref{thm:1} is determined by   the particular choice of the range $(1/4,/1/2)$ for the parameter $\delta$, which  guarantees that $1-\delta >\delta$.

The method of approximate spectral projection operators is based on variational arguments. Thus, if $S$ is a symmetric, lower semi-bounded operator in a separable Hilbert space, and  if $V$ is a subspace contained in its domain $ \mathcal{D}(S)$,   the variational quantity
\bqn
\label{eq:15}
\Ncal(S,V)=\sup_{L\subset V}\{\dim L: ( S \,u,\,u) <0 \quad \forall \,\, 0\not=u \in \L\}
\eqn
can be used to give a qualitative description of the spectrum of $S$.
In our case one has
$$ N_\chi(\lambda)=\Ncal(A_0-\lambda \1, \H_\chi \cap \CT({\bf{X}})).$$
Indeed,  the Friedrichs extension of $\mathrm{res} \circ A_0 \circ \mathrm{ext}: \CT({\bf{X}}) \cap \H_\chi \longrightarrow \mathrm{res}\,  \H_\chi$ is given by $A_\chi$, and the assertion follows with \cite{levendorskii}, Lemma A.2.
Now by  the general theory of compact, self-adjoint operators, zero is the only accumulation point of the point spectra of $\E_\lambda$ and $\F_\lambda$, as well as the only point that could possibly belong to the continuous spectrum. Therefore the number of eigenvalues of $\E_\lambda$ which are $\geq 1/2$, and whose eigenfunctions belong to the isotypic component $\H_\chi$ is clearly finite, and shall be denoted by $N^{\E_\lambda}_\chi$. Similarly, the  number of eigenvalues of the operators $\F_\lambda$
which are $\geq 1/2$, and whose eigenfunctions belong to the isotypic component $\H_\chi$,   shall be denoted by $N^{\F_\lambda}_\chi$. As it was shown in Part I, Theorems 5 and 6, these quantities constitute upper and lower bounds for  the spectral counting function $N_\chi(\lambda)$, namely
\bqn
 N^{\E_\lambda}_\chi -C \leq  \Ncal(A_0-\lambda \1, \H_\chi \cap \CT({\bf{X}})) \leq M^{\F_\lambda}_\chi +C
\eqn
for some constant $C>0$.
Furthermore, by Lemmata 11 and 12  of Part I one has
\begin{align*}
2 \tr (P_\chi \E_\lambda \cdot P_\chi \E_\lambda) -\tr P_\chi \E_\lambda -c_1 &\leq N_\chi^{\E_\lambda} \leq 3 \tr P_\chi \E_\lambda -2 \tr (P_\chi \E_\lambda\cdot P_\chi \E_\lambda) +c_2,\\
2 \tr (P_\chi \F_\lambda \cdot P_\chi \F_\lambda) -\tr P_\chi \F_\lambda -c_1 & \leq N_\chi^{\F_\lambda} \leq 3 \tr P_\chi \F_\lambda -2 \tr (P_\chi \F_\lambda \cdot P_\chi \F_\lambda) +c_2,
\end{align*}
for some constants $c_i>0$. The study of the asymptotic behaviour of $N_\chi(\lambda)$ is therefore reduced to an examination of the traces of   $ P_\chi \E_\lambda $ and $ P_\chi \F_\lambda$, together with their squares, and  will occupy us for the rest of this paper.

\section{Compact group actions and the  principle of the stationary phase}

In this section, we shall begin to estimate the traces of $P_\chi\E_\lambda$ and $P_\chi \F_\lambda$ using the method of the stationary  phase, in order to obtain a description of the spectral counting function $N_\chi(\lambda)$ as $\lambda \to +\infty$. As  mentioned in the introduction, first order asymptotics for invariant elliptic operators were  already obtained in \cite{donnelly, bruening-heintze} in the general case of effective group actions by using heat kernel methods; nevertheless, estimates for the remainder are  not accessible via this approach. On the other hand, the derivation of remainder estimates within the framework of Fourier integral operators or, as we  shall see, within the setting of approximate spectral  projections, meets with serious difficulties when singular orbits are present. The reason for this is that, using these approaches, one is led to the study of the asymptotic behavior of integrals of the form
\bq
\label{eq:equivSPint}
\int _G\int_{\rn} \int_{\R^n}  e^{i (x-kx) \xi/\mu }  a(x,\xi,k) dx \, \dbar \xi dk, \qquad \mu \to 0^+,
\eq
via the generalized stationary phase theorem,  where $a(x,\xi,k)\in \CT(\rn \times \rn \times G)$ is an amplitude which might also depend on $\mu$. While for free group actions, the critical set of the phase function $(x-kx) \xi$ is a smooth manifold, this is no longer the case for general effective actions, so that, a priori, the principle of the stationary phase can not be applied in this case. Nevertheless, in what follows, we shall show how to circumvent this obstacle by partially resolving the singularities of the critical set of the phase function in question, and in this way obtain remainder estimates for $N_\chi(\lambda)$  in the case of singular group actions.  Let us begin by stating the generalized stationary phase theorem.
\begin{theorem}[Generalized stationary phase theorem for manifolds]
\label{thm:SP}
Let $M$ be a  $n$-dimensional Riemannian manifold,  $\psi \in \Cinft(M)$ be a real valued phase function, $a \in \CT(M)$, $\mu >0$, and set
\bqn
I({\mu})=\int_M e^{i\psi(m)/\mu} a(m) \, dm,
\eqn
where $dm$ denotes the volume form on $M$. Let $\mathcal C=\mklm{m \in M: \psi':TM_m \rightarrow T\R_{\psi(m)} \text{ is zero}}$ be the  critical set of the phase function $\psi$, and assume that
\begin{enumerate}
\item[(i)] $\mathcal{C}$ is a smooth submanifold of $M$  of dimension $p$ in a neighborhood of the support of $a$;
\item[(ii)]for all $m \in \mathcal{C}$, the restriction $\psi''(m)_{|N_m\mathcal{C}}$ of the Hessian of $\psi$ at the point $m$  to the normal space $N_m\mathcal{C}$ is a non-degenerate quadratic form.
\end{enumerate}
\noindent
Then, for all $N \in \N$, there exists a constant $C_{N,\psi}>0$ such that
\bqn
|I(\mu) - e^{i\psi_0/\mu}(2\pi \mu)^{\frac {n-p}{2}}\sum_{j=0} ^{N-1} \mu^j L_j (\psi;a)| \leq C_{N,\psi} \mu^N \vol (\supp a \cap \mathcal{C}) \sup _{l\leq 2N} \norm{D^l a }_{\infty,M},
\eqn
where $D^l$ is a differential operator on $M$ of order $l$, and $\psi_0$ is the constant value of $\psi$ on $\mathcal{C}$. Furthermore, for each $j$ there exists a constant $\tilde C_{j,\psi}>0$ such that 
\bqn
|L_j(\psi;a)|\leq \tilde C_{j,\psi}  \vol (\supp a \cap \mathcal{C}) \sup _{l\leq 2j} \norm{D^l a }_{\infty,\mathcal{C}},
\eqn
and, in particular,
\bqn
L_0(\psi;a)= \int _{\mathcal{C}} \frac {a(x)}{|\det \psi''(m)_{|N_m\mathcal{C}}|^{1/2}} d\sigma_{\mathcal{C}}(m) e^{ i \pi\sigma_{\psi''}},
\eqn
where $\sigma_{\psi''}$ is the constant value of the signature of $\psi''(m)_{|N_m\mathcal{C}}$ for $m$ in $\mathcal{C}$.
\end{theorem}
\begin{proof}
See \cite{hoermanderI,combescure-ralston-robert}.
\end{proof}

From now  on, we shall restrict ourselves to the study of $\tr  P_\chi \E_\lambda$, since the corresponding considerations for $\F_\lambda$ are completely analogous. Let therefore $\sigma^l(\E_\lambda)(x,\xi)$ denote the left symbol of $\E_\lambda$. Since $\sigma^l(\E_\lambda)$ is $G$-invariant, we have
 \bqn
 P_\chi \E_\lambda u(x)=  {d_\chi} \int_G\int \int \overline{\chi(k)} e^{i( x -k y ) \xi} \sigma^l(\E_\lambda) ( x, \xi ) u(y)  dy \, \dbar \xi \, dk, \qquad   u \in \CT(\R^n).
 \eqn
 The kernel of $P_\chi\E_\lambda$, which is a rapidly decreasing function, is   given by the absolutely convergent integral
 \bqn
 K_{P_\chi \E_\lambda}(x,y)=  {d_\chi}\int_G \int  \overline{\chi(k)}e^{i( x -k y ) \xi} \sigma^l(\E_\lambda) ( x, \xi)    \, \dbar \xi\, dk.
 \eqn
Consequently, the trace of $P_\chi \E_\lambda$ can  be computed by
 \begin{gather*}
 \tr P_\chi \E_\lambda=\int K_{P_\chi \E_\lambda}(x,x) dx
 =d_\chi\int_G\int \int  \overline{\chi(k)} e^{i(x- k x) \xi} \sigma^l(\E_\lambda) (x, \xi ) dx \, \dbar \xi dk.
 \end{gather*}
 As already noticed, the decay properties of $\sigma^l(\E_\lambda) \in S(h^{-2\delta}g,1)=\Gamma_{1-\delta,\delta}^0(\rnn)$ are independent of $\lambda$ , while its support does depend on $\lambda$.  Indeed, as it was already explained in Part I,  equation (51),
 \bq
\label{eq:23b}
\sigma^l(\E_\lambda) = (\eta_{\lambda,-2}^2 \chi_\lambda)^2 ( 3- 2 \eta_{\lambda,-2}^2\chi_\lambda) + f_\lambda + r_\lambda,
\eq
where $r_\lambda \in S(h^{-2\delta}g,h^{N(1-2\delta)})$ for arbitrary large $N$, and $ f_\lambda \in S(h^{-2\delta} g, h^{1 - 2\delta})$,  everything uniformly in $\lambda$. Moreover, in Lemma \ref{lemma:11} we shall see that
\bqn
\supp f_\lambda\subset A_{c,\lambda}=\mklm{(x,\xi) \in {\bf{X}} \times \rn: a_\lambda < c(h^{\delta-\omega}+d)}.
\eqn
 Now, since $|r_\lambda(x,\xi)| \leq C'(1+|x|^2+|\xi|^2)^{-N/2}$ for some constant $C'$ independent of $\lambda$ and $N$ arbitrarily large, we get the uniform bound
\bdm
\int \int |r_\lambda(x,\xi)| dx \, \dbar \xi \leq C;
\edm
note that the $x$-dependence of $h(x,\xi)$ is crucial at this point. In order to determine the asymptotic behaviour of  $\tr P_\chi \E_\lambda$ with respect to $\lambda$, we can therefore neglect  the contribution coming from $r_\lambda(x,\xi)$, so that
 \begin{gather*}
 \tr P_\chi \E_\lambda=
d_\chi\int_G\int \int  \overline{\chi(k)} e^{i(x-k x) \xi}[(\eta_{\lambda,-2}^2 \chi_\lambda)^2 ( 3- 2 \eta_{\lambda,-2}^2\chi_\lambda) + f_\lambda]dx \, \dbar \xi dk +O(1),\\
 \end{gather*}
as  $\lambda$ goes to infinity. To apply the generalized stationary phase theorem, we introduce the new parameter
 \bqn
 \mu=\lambda^{-1/2m}, \qquad \lambda=\mu^{-2m},
 \eqn 
 and performing the  change of variables
 $$\Psi_\mu:(x,\xi)\mapsto (x,\mu\xi)$$
 we obtain
$$\tr P_\chi \E_\lambda=d_\chi\lambda^{n/2m} I(\lambda^{-1/2m})+O(1), $$
where we set
\begin{align}
\label{eq:15b}
\begin{split}
I(\mu)&=\int _G\int_{\bf X} \int_{\R^n}  e^{\frac{i}{\mu} \psi(x,\xi,k) }  \overline{\chi(k)}\sigma_\mu(x,\xi) dx \, \dbar \xi dk,\\
\psi(x,\xi,k)&=( x- k x) \xi,\\
\sigma_\mu &=[(\eta_{\lambda,-2}^2 \chi_\lambda)^2 ( 3- 2 \eta_{\lambda,-2}^2\chi_\lambda) + f_\lambda]\circ \Psi_\mu^{-1}.
\end{split}
\end{align}
As we shall see later,  there exists a compact subset $K \subset \R^{2n}$, such that $\sigma_\mu$ has  support in $K$ for all $\mu>0$, see \eqref{eq:suppsigma}. To get an asymptotic expansion of $I(\mu)$ as $\mu\to 0^+$  via the generalized stationary phase theorem, we commence by examining  the critical set
 \bq
 \mathcal C=\mklm{ (x,\xi,k) \in {\bf X}\times\mathbb{R}^{n}\times G: \psi' (x,\xi,k) =0}
 \eq
  of the phase function $\psi$. After a straightforward  computation we obtain
  \bqn
  \mathcal{C}=\mklm{(z,k)\in \Omega_0 \times G: kz=z},
   \eqn
    where we put $z=(x,\xi)$, and
\bqn
  \Omega_0=\mklm{(x,\xi) \in {\bf{X}}\times\R^{n}: \eklm {Ax,\xi}=0 \text{ for all } A \in \g},
\eqn
$\g$ being the Lie algebra of $G$. $\eklm{\cdot,\cdot}$ denotes the Euclidean product in $\rn$.
Note that $\Omega_0$ is invariant under the  Hamiltonian action of $G$ on the cotangent space $T^\ast(\bf{X})$ given by $(x,\xi) \mapsto (kx,k\xi)$, as well as homogeneous with respect to $x$ and $\xi$. It has the following interpretation in terms of the Hamiltonian action of $G$ on $T^\ast({\bf{X}})$. If $(A_1,\dots, A_d)$ is a basis of $\g$,  let
\bqn
\mathbb{J}:T^\ast({\bf{X}})\simeq {\bf X}\times \rn\to \g\simeq \R^d, \quad (x,\xi) \to (\eklm{A_1x,\xi},\dots,\eklm{A_dx,\xi}),
\eqn
 be the associated momentum map,  and denote by
 \bqn
 \Omega_0/G=\mathbb{J}^{-1}(\{ 0 \})/G
 \eqn
the symplectic quotient of $T^\ast({\bf{X}})$ at level zero.  This quotient is naturally related to the critical set of the phase function  in question, and  we shall prove that   $N_\chi(\lambda)$ is asymptotically determined by  a certain volume of the quotient $\Omega_0/G$. Now, the major difficulty in applying the generalized stationary phase theorem in our setting  stems from the fact that, due to the singular orbit structure of the underlying group action,  the zero level $\Omega_0$ of the momentum map, and, consequently, the considered critical set $\mathcal{C}$, are in general singular varieties. In fact,  if the $G$-action on $T^\ast({\bf X)}$ is not free, the considered momentum map is no longer a submersion, so that $\Omega_0$ and $\Omega_0/G$ are not smooth anymore. 
To circumvent this difficulty, we will partially resolve the singularities of  $\mathcal C$ by   constructing a partial resolution of $\Omega_0$, which takes into account the singular orbit structure of the underlying $G$-action, and then apply the stationary phase theorem in the resolution space.
\footnote{As we shall see in Section 5, 
$$\sigma_\mu(x,\xi){\longrightarrow} \1_{\{a_{2m}\leq 1\}}(x,\xi) \quad  \qquad \text{as } {\mu\to 0^+},$$
where $\1_A$ stands for the characteristic function of the set $A$. By homogeneity, $a_{2m}(0,0)=0$, so that zero is contained in the support of $\1_{\{a_{2m}\leq 1\} }$. In general, $\sigma_\mu$ is therefore not supported away from the set of singular points of $\mathcal C$, since $(0,0)$ is always a singularity of $\Omega_0$ in case that $0 \in {\bf{X}}$.}

In what follows, we shall therefore briefly recall some basic  notions of the theory of compact group actions. For a detailed exposition,  we refer the reader  to \cite{bredon} or \cite{kawakubo}. Let $G$ be a compact Lie group acting locally smoothly on some $n$-dimensional $\Cinft$-manifold $M$, and denote the stabilizer, or isotropy group, of $x\in M$ by
$$G_x=\{ k\in G\, :\, k\cdot x=x \}.$$
The orbit of a point $x \in M$  under the action of $G$ will be denoted by $\Ocal_x$.
Assume that $M/G$ is connected.  
One of the main results in the theory of compact group actions is the following
\begin{theorem}[Principal orbit theorem]
There exists a maximum orbit type $G/H$ for $G$ on $M$. The union $M_{(H)}$ of orbits of type $G/H$ is open and dense, and its image in $M/G$ is connected.
\end{theorem}
\begin{proof}
See \cite{bredon}, Theorem IV.3.1.
\end{proof}
Orbits of type $G/H$ are called of principal type, and the corresponding isotropy groups are called principal. A principal isotropy group  has the property that it is  conjugated to a subgroup of each stabilizer of $M$. The following result says that there is a stratification of the considered $G$-space into orbit types.
\begin{theorem}
Let $K$ be a subgroup of $G$, and denote the set of points on orbits of type $G/K$ by $M_{(K)}$. Then $M_{(K)}$ is a topological manifold, which is locally closed. Furthermore, $\overline{M_{(K)}}$ consists of orbits of type less than or equal to type $G/K$. The orbit map $M_{(K)} \rightarrow M_{(K)}/G$ is a fiber bundle projection with fiber $G/K$ and structure group $N(K)/K$.
\end{theorem}
\begin{proof}
See \cite{bredon}, Theorem IV.3.3.
\end{proof}
Let now $M_\tau$ denote the union of non-principal orbits of dimension at most $\tau$.
\begin{theorem}
If $\kappa$ is the dimension of a principal orbit, then $\dim M/G =n-\kappa$, and  $M_\tau$  is a closed set of dimension at most $n-\kappa+\tau-1$.
\end{theorem}
\begin{proof}
See \cite{bredon}, Theorem IV.3.8.
\end{proof}
Here the dimension of $M_\tau$ is understood in the sense of general dimension theory.
In what follows, we shall write $\Sing M= M - M_{(H)}=M_\kappa$. Clearly,
\bqn
\Sing M= M_0 \cup (M_1-M_0) \cap (M_2-M_1) \cup  \dots \cup(M_\kappa -M_{\kappa -1}),
\eqn
where $M_i-M_{i-1}$ is precisely the union of non-principal orbits of dimension $i$, and $M_{-1}=\emptyset$, by definition. Note that 
\bqn
M_i-M_{i-1}=\bigcup_{j} M_{(H^i_j)}, \qquad \dim G/H^i_j =i, 
\eqn
is a disjoint union of topological manifolds of possibly different dimensions.
We apply this theory now to the case where $M=\rn$, and $G$ is a compact subgroup of $\O(n)$.
\begin{definition}\label{def2}
Let $G/H_0$ be the  principal orbit type of the action of $G\subset \O(n)$ on $\rn$, and denote by $\kappa$ the dimension of $G/H_0$.
\end{definition}
 Since ${\bf X}$ is open in $\rn$, it has the same principal orbit type than $\rn$.
Now, even if $\Omega_0$ is not a smooth manifold, it can be shown that it has a Whitney stratification into smooth submanifolds, see \cite{ortega-ratiu}, Theorem 8.3.1, which corresponds to the stratification of $T^\ast({\bf{X}})$ and $\rn$ into orbit types. In particular, the strata of $\Omega_0$ are submanifolds of $\rnn$, and $\Omega_0$ admits a principal orbit type, too.
\begin{proposition}\label{def3}
Let $\Reg \Omega_0=\Omega_{0_{(H_1)}}$ be the principal stratum of $\Omega_0$. Then
$\Reg \Omega_0$ is an open dense subset of $\Omega_0$, and a submanifold of ${\bf X} \times\rn$ of codimension $\kappa$. Moreover, for $z \in \Reg \Omega_0$ one has
\bq\label{roch1}
 T_z(\Reg \Omega_0)=(J\g z)^\perp, \qquad
 \mbox{ where }\quad
J=\left(
\begin{array}{cc}
0    & \1_n\\
-\1_n & 0
\end{array}\right).
\eq
Futhermore, $H_1$  is conjugated to $H_0$, and thus
\bqn
\Reg \Omega_0=\mklm{ z \in \Omega_0: G_z \text{ is conjugated to }H_0}.
\eqn
In particular,  if $(x,\xi)\in\Omega_0$,  and if $\Ocal_x$ or $\Ocal_\xi$ are of type $G/H_0$, then $(x,\xi)\in \Reg \Omega_0$.
\end{proposition}
To prove the proposition, we need the following
\begin{lemma}\label{isotrop}
Assume that $(x,\xi) \in \Omega_0$.
If  $\Ocal_x$ is of principal orbit type in $\rn$, then $G_x \subset G_\xi$.\\
 If  $\Ocal_\xi$ is of principal orbit type in $\rn$, then $G_\xi \subset G_x$.
\end{lemma}
\begin{proof}
Let $(x,\xi) \in \Omega_0$, that is, $\xi \in N_x \Ocal_x$, where $N_x\Ocal_x$ denotes the normal space to the $G$-orbit $\Ocal_x$ at the point $x$, which is a vector subspace in $\rn$. Assume now that $\Ocal_x$ is of principal type. Denote by $V_\epsilon$ the open $\epsilon$-ball in $N_x\Ocal_x$, and consider the linear tube
\bqn
G \times_{G_x} V_\epsilon \longrightarrow G \cdot V_\epsilon, [g,v] \to gv,
\eqn
around $\Ocal_x$, see \cite{bredon}, Corollary II.5.2. By loc.\  cit., Theorem IV.3.2, $G_x$ acts trivially on $V_\epsilon$, and consequently also on $N_x\Ocal_x$, and the assertion follows. To see this directly, one can also argue as follows. Let $(x,\xi)\in\Omega_0$, so that  $\xi\in (\g x)^\perp$. If $g\in G_x$, then $g\xi \in (\g x)^\perp$. Thus $(g-\1)\xi\in(\g x)^\perp$. We claim that if $\Ocal_x$ is of principal orbit type in $\rn$, then $(g-\1)\xi\in \g x$, which will yield  $(g-\1)\xi=0$, and prove the inclusion  $G_x\subset G_\xi$. Now, by \cite{kawakubo}, Theorem 4.19, the canonical projection $\pi :  \rn_{(H_0)}\twoheadrightarrow \rn_{(H_0)}/G$ is a smooth submersion. Since the preimage of the tangent space of a smooth manifold under a submersion is equal to the tangent space of the preimage of the considered manifold at the given point,  $\ker d_x\pi=\g x$. Moreover, since $M_{(H_0)}$ is an open set of $\rn$, one can differentiate the relation
$$\pi(gy)=\pi(y),$$
with respect to variable $y$ at $x$ to obtain $d_{gx}\pi\circ g=d_x\pi$. Since $gx=x$, $d_x\pi\circ(g-\1)=0$. This proves that the image of $(g-\1)$ is contained in $\ker d_x\pi=\g x$.
\end{proof}
\begin{proof}[Proof of Proposition \ref{def3}]
The first part of the statement follows from the references previously given, while the characterization of the tangent space is obtained by observing that  $\dim \Reg \Omega_0=2n-\kappa$. By the previous lemma,  $(\rn\times \rn_{(H_0)})\cap \Omega_0$ is a non-empty open subset of $\Omega_0$ consisting of orbits of type $G/H_0$. As $\Reg \Omega_0$ is open and dense in $\Omega_0$, it must intersect $(\rn\times \rn_{(H_0)})\cap \Omega_0$, and therefore consist of orbits of type $G/H_0$.
\end{proof}
In what follows, we will denote by $ \Sing \Omega_0$ the complement of $ \Reg \Omega_0$ in $ \Omega_0$.
The next lemma will provide us with a suitable parametrization of $\Reg \Omega_0$.
\begin{lemma}
 \label{lemma:parametr}
The sets $ \{(x,\xi) \in \Reg \Omega_0: x    \in \Sing \rn\}$, $ \{(x,\xi) \in \Reg \Omega_0: \xi    \in \Sing \rn\}$  have measure zero in  $\Reg \Omega_0$ with respect to the induced volume form on $\Reg \Omega_0$.
\end{lemma}
\begin{proof}
We shall show that $N=\mklm{(x,\xi) \in \Omega_0: x \in \Sing \rn}$ is a closed set in $\Omega_0$ of dimension at most $2n-\kappa-1$. Indeed, with $M=\rn$, and  notations as above,
\begin{align*}
N&= \bigcup_{i=0}^\kappa \mklm {(x,\xi) \in \Omega_0: x \in M_i-M_{i-1}}= \bigcup_{i=0}^\kappa\bigcup _{j(i)} \mklm {(x,\xi) \in \R^{2n}: x \in \rn_{(H^i_j)}, \xi \in N_x \Ocal_x},
\end{align*}
where the union over $j(i)$ ranges over all non-principal orbit types $G/H^i_j$ with $\dim G/H^i_j=i$. By the previous theorem, $\dim \rn_{(H^i_j)} \leq \dim M_i \leq n-\kappa +i-1$,
and in addition, $\dim N_x\Ocal_x =n-i$ for all $x \in  \rn_{(H^i_j)}$. Consequently, $ \mklm {(x,\xi) \in \R^{2n}: x \in \rn_{(H^i_j)}, \xi \in N_x \Ocal_x}$ is a subset of $\Omega_0$ of dimension at most $2n-\kappa-1$. Since for orthogonal group actions there are only finitely many orbit types, the union over $j(i)$ is finite, and the assertion of the lemma follows. 
\end{proof}
Finally, for future reference we note the following
\begin{lemma}\label{critset}
  The set
    \bqn
    \Reg {\mathcal{C}}=\mklm{(z,k)\in  \Reg \Omega_0  \times G: kz=z}
    \eqn
    is a smooth submanifold of dimension $2n+d-2\kappa$, and for $(z,k)\in \Reg {\mathcal{C}}$,
    \bqn
   T_{(z,k)} \Reg {\mathcal{C}}=\mklm{(\alpha,Ak):\alpha\in T_z\Reg \Omega_0, A\in\mathcal{G} \mbox{ and } (\1-k)\alpha+Az=0 }.
   \eqn
\end{lemma}
\begin{proof}
See \cite{cassanas}, Lemma 3.2.
\end{proof}
 In particular note that if $(z,k)$ belongs to  $\Sing {\mathcal{C}}$, the complement of $ \Reg {\mathcal{C}}$ in $\mathcal{C}$, then $z$ must necessarily lie in $\Sing \Omega_0$.
 After these preliminary remarks, we are now ready for the analysis  of $I(\mu)$.

\section{Phase analysis  and partial desingularization}

We shall now start  with the computation of an asymptotic formula for $I(\mu)$ via the generalized stationary phase theorem by partially resolving the singularities of the critical set 
\bqn
\mathcal C=\mklm{ (x,\xi,k) \in {\bf X}\times\mathbb{R}^{n}\times G: \psi' (x,\xi,k) =0}
\eqn
of the phase function $\psi(x,\xi,k)=\eklm{x-kx,\xi}$. Such a resolution will be given by a  proper $\R$-analytic map $\zeta:\tilde M \rightarrow M$ of some smooth manifold $\tilde M$ onto $M=\rn$, inducing a  transformation $\zeta:\tilde {\mathcal{C}} \rightarrow {\mathcal{C}}$ such that
$\tilde {\mathcal{C}}$ is a partially desingularized subvariety of $\tilde M$, and $\zeta$ induces an isomorphism of real analytic manifolds $\zeta^{-1}(\Reg {\mathcal{C}})\rightarrow {\Reg \mathcal{C}}$, where $\Reg {\mathcal{C}}$ denotes
the set of nonsingular points of ${\mathcal{C}}$.
By performing such a resolution we will be led  to a new phase function, whose critical set is no longer a singular variety. As before, denote by $\rn_{(H_0)}$ the union of all orbits of principal type $G/H_0$ in $\rn$.
We will  construct an explicit resolution of $\tilde {\mathcal{C}}$ by constructing a resolution of $\Omega_0$ first, 
under the following
\begin{assumption}\label{assump1}
  The set $\Sing \rn=\rn\setminus \rn_{(H_0)}$ is included in a strict vector subspace $F$ of $\rn$ of dimension $r<n$.
\end{assumption}
\begin{remark} Particular cases of Assumption \ref{assump1}  are

{\it i) Transitive actions on the sphere}. For any   compact subgroup of $O(n)$ acting transitively on the $(n-1)$-dimensional sphere,   $\Sing \rn=\{0\}$. The list of compact, connected Lie groups acting transitively and effectively on spheres has been found by Montgomery and Samelson \cite{montgomery-samelson}. It includes  all the holonomy groups of a simply-connected Riemannian manifold with an irreducible, nonsymmetric metric appearing in Berger's list, and in particular,  the group $\SO(n)$ acting on $\rn$.

 {\it ii) Cylindrical actions}. For the group of rotations around an axis in $\R^n$, $\R^n_{\textrm{sing}}$ is equal to the rotation axis. More generally, any group conjugated to $G\times \{ \1_q\}$  in $\O(n)$, where $G$ is a compact subgroup of $\O(p)$ acting transitively on the $(p-1)$-dimensional sphere, and $p+q=n$, is included.
\end{remark}
We begin by considering the blowing-up of $M=\R^{2n}$ with center $C=\mklm{\xi_1=\dots= \xi_n=0}$ given by
 \bqn
 \tilde M= \{ (x,\xi,[\mu]) \in M \times \R\mathbb{P}^{n-1}: \xi_i \mu_j=\xi_j \mu_i, \quad i <j\},
 \eqn
 together with  the  monoidal transformation
\bqn
\zeta_M:\tilde M \longrightarrow M, \quad (x,\xi,[\mu]) \mapsto (R_0 x,R_0 \xi),
\eqn
with  $R_0 \in \O(n)$ such that
$$R_0(\R^r\times \{0\})=F.$$

Covering $\tilde M$ with the charts $\tilde M_j=\tilde M\cap (M\times U_j)$, where $U_j=\mklm{[\mu] \in \mathbb{RP}^{n-1}:\mu_j \not=0}$, one obtains in $\tilde M_j$ the local coordinates
\bqn
x_i, \quad   i=1,\dots, n, \qquad \eta_k=\frac{\mu_k}{\mu_j}, \quad \eta_j=\xi_j, \quad k=1, \stackrel{\wedge} \dots ,n,
\eqn
 and we write
 \bqn
 \tilde \phi_j: \R^{2n} \longrightarrow \tilde M_j, \quad  (x, \eta) \mapsto (x,\eta_j(\eta_1, \dots,1,\dots, \eta_n), [\eta_1:\dots:1:\dots:\eta_n]).
 \eqn
 Now, the total transform of $\Omega_0$ is given by $\tilde \Omega_0^{\mathrm{tot}}=\zeta_M^{-1}(\Omega_0)$, and contains the exceptional divisor $E=\zeta_M^{-1}(C)$,  while the strict transform of $\Omega_0$ in the $j$-th chart is   locally given by
 \bqn
 \tilde \Omega_0^{\mathrm{st}}=\mklm{(x,\eta) \in \R^{2n}: \eklm{AR_0x,R_0(\eta_1,\dots, 1,\dots, \eta_n)}=0, \, A \in \g  }.
 \eqn
For $j=r+1,\dots, n$, it is a non-singular variety, since in this case the condition $(x,\eta) \in \tilde \Omega_0^{\mathrm{st}}$ implies that 
 $(R_0x,R_0(\eta_1,\dots,1,\dots,  \eta_n))\in \Reg \Omega_0$ by Assumption 1, and Proposition  \ref{def3}.
By functoriality, the $G$-action on $M$ lifts to a $G$-action on $\tilde M$.
 To  construct a partial resolution for  $\mathcal C$, we put $N=M\times G$, $\tilde N=\tilde M \times G$, and $\zeta_{N}:\tilde N \rightarrow N, (x,\xi,[\mu],k) \mapsto (x,\xi,k)$. Using the coordinates introduced above, we see that the strict transform of $\mathcal C$ with respect to $\zeta_N$ is locally given by
 $$\tilde{\mathcal{C}}^{\mathrm{st}}=\{(x,\eta,k) \in \tilde \Omega_0^{\mathrm{st}} \times G: (k-\1)R_0x=0, \, (k -\1)R_0(\eta_1, \dots, 1,\dots, \eta_n)=0\}.$$
For $j=r+1,\dots,n$, $G$ acts on $\tilde \Omega_0^{\mathrm{st}}$ only with one orbit type, so that in this case $\tilde {\mathcal C}^{\mathrm{st}}$ must be non-singular. Let now $I(\mu)$ be defined as in \eqref{eq:15b}.
 Since each chart $\tilde M_j$ completely covers $\tilde M$ except for a set of measure zero, one has
\begin{align}
\label{eq:16}
I(\mu) &=  \int _G \int_{\R^{2n}} e^{i \tilde \psi_j (x,\eta,k) /\mu}  \tilde\sigma_{\mu,j}(x,\eta)  \overline \chi (k) |\eta_j^{n-1}| \, dx\dbar\eta \, dk
\end{align}
for arbitrary $j$, where we put $\tilde \psi_j(x,\eta,k)=\psi ((\zeta_M\circ \tilde \phi_j)(x,\eta),k)$,  $\tilde \sigma_{\mu,j} (x,\eta)=(\sigma_\mu \circ \zeta_M\circ \tilde \phi_j)(x,\eta)$, and took into account the fact that  $|\det D(\zeta_M \circ \tilde \phi_j) (x,\eta)|=|\eta_j^{n-1}|$.
In what follows, we shall  work in the chart $j=n$, and denote  $\tilde  \psi_n$ and $\tilde \sigma_{\mu,n}$ simply by $\tilde \psi$ and $\tilde \sigma_\mu$, respectively. 
Let us now introduce the new parameter\footnote{The idea of introducing the new parameter $\nu$ was taken from \cite{helffer-elhouakmi}, Section 6. Nevertheless, Helffer and El-Houakmi work in spherical variables, which leads to secondary critical points that were not explicitly taken into account in their work. Our approach does not lead to secondary critical points.}
\bqn
\nu=\mu/\eta_n.
\eqn
Defining the new phase function\footnote{The subscript 'wk' stands for ``weak transform''.}
\bqn
 \phr:\rn\times \R^{n-1}\times G\to \R, \quad  \phr(x,\eta',k)=\eklm{(\1-k)R_0x;R_0(\eta',1))}_{\rn}, 
\eqn
and taking into account \eqref{eq:suppsigma}, we  write
 \bqn
 I(\mu) =\frac 1 {(2\pi)^n} \int_{-E_0}^{E_0} I_{\eta_n}(\mu/\eta_n)|\eta_n|^{n-1}  \d\eta_n,
 \eqn
 where $E_0$ is some suitable positive number, and 
 \bq
 \label{eq:17}
 I_{\eta_n} (\nu) =  \int _G \int_{\R^{2n-1}} e^{ \frac{i}{\nu} \phr(x,\eta',k)} \tilde\sigma_{\nu\eta_n}(x,\eta',\eta_n)  \overline \chi (k)  \, dx\, d\eta' \, dk.
 \eq
The significance of the new phase function $\psi_{wk}$ stems from the following proposition. It will enable us to derive an asymptotic formula for $I_{\eta_n} (\nu)$ as $\nu$ goes to zero by using the stationary phase theorem in the region where $\eta_n$ is not small. Note that, in particular,  Theorem \ref{thm:SP} will allow us to handle the dependence of the amplitude $\tilde \sigma_\mu$ in variable $\mu=\nu\eta_n$.
  \begin{proposition}\label{phaseanalysis}
Let $\mathcal{C}_{\phr}=\{\phr'=0\}$ denote the critical set of $\phr$. Then
$$ \mathcal{C}_{\phr}=\{(x,\eta',k)\in \rn\times \R^{n-1}\times G : (R_0x,R_0(\eta',1),k)\in \Reg \mathcal{C} \}.$$
It is a smooth submanifold of $\rn\times \R^{n-1}\times G$ of codimension $2\kappa$. Moreover, at each point $(x,\eta',k) $ of $\mathcal{C}_{\phr}$, the transversal Hessian of $\phr$ defines a non-degenerate quadratic form on the normal space $N_{(x,\eta',k)}\mathcal{C}_{\phr}$ of $\mathcal{C}_{\phr}$ in $\rn\times \R^{n-1}\times G$.
 \end{proposition}
 \begin{remark}
 Note that if $\psi_{wk}$ is regarded as a function on $\tilde N$, that is, as a function of  $x,\eta$, and $k$, the proposition implies that its critical set is given by the strict transform $\tilde{\mathcal{C}} ^{\mathrm{st}}$ of $\mathcal C$; moreover, its transversal Hessian does not degenerate along $\tilde{\mathcal{C}} ^{\mathrm{st}}$.
 \end{remark}
  \begin{proof}[Proof of Proposition \ref{phaseanalysis}]
 We shall denote by $(e_1,\dots,e_n)$ the canonical basis of $\rn$.\
 With respect to the coordinates $(x,\eta,k)$ one computes
$$\left\{
\begin{array}{l}
\partial_x\phr(x,\eta',k)=0\iff (\1-k^{-1})R_0(\eta',1)=0.\\
\partial_k\phr(x,\eta',k)=0\iff \eklm{AR_0x,R_0(\eta',1)}=0, \; \forall A\in\g.\\
\partial_{\eta'}\phr(x,\eta',k)=0 \iff \eklm{ (\1-k)R_0x,R_0e_i}=0, \; i=1,\dots,n-1.
\end{array}
\right.$$
The second equation is equivalent to the fact that $(R_0x,R_0(\eta',1))\in \Omega_0$. By Assumption \ref{assump1}, $R_0(
\eta',1)\notin F$, so that using  Proposition \ref{def3}, we obtain that our second equation is equivalent to the fact that $(R_0x,R_0(\eta',1))\in \Reg \Omega_0$.  Using Lemma \ref{isotrop}, the two first equations imply that $kx=x$, and therefore imply the third one. Consequently, we obtain
$$ \mathcal{C}_{\phr}=\{(x,\eta',k)\in \rn\times \R^{n-1}\times G : (k-\1)R_0x=0,\; (k-\1)R_0(\eta',1)=0, \; (R_0x,R_0(\eta',1))\in \Reg \Omega_0\}.$$
Next, we see immediately that $ \mathcal{C}_{\phr}$ is diffeomorphic to the intersection of  $\Reg {\mathcal{C}}$ and $(R_0\times R_0)(\{\eta_n=1\})$. Thus, in order to show  that $\mathcal{C}_{\phr}$ is a smooth manifold, we  have to prove that these two sets are transversal. Let $(z,k)=(R_0x,R_0\eta,k)\in\Reg {\mathcal{C}}\cap(R_0\times R_0)(\{\eta_n=1\}) $. We need to prove that $T_{(z,k)}\Reg {\mathcal{C}}\not\subset (R_0\times R_0)(\{\eta_n=0\})$. For this purpose, consider $\alpha=(-R_0x,R_0\eta)$. This is an element of $T_{z}\Omega_0=J\g z^\perp$ which satisfies $(R_0\times R_0)^{-1}(\alpha)\not\in\{\eta_n= 0\}$. Moreover, we shall see later in  Lemma \ref{formulaint} that $kz=z$ implies $(k-\1)\alpha \in \g z$  for  all $ \alpha \in T_z\Omega_0$. Consequently,  there exists  an $A\in \g$ such that $(\alpha,Ak)\in T_{(z,k)}\Reg {\mathcal{C}}\setminus(R_0\times R_0)(\{\eta_n=0\})$. The dimension of $ \mathcal{C}_{\phr}$ follows from Lemma \ref{critset}, and the tangent space at $(x,\eta',k)$ is  therefore given by
\bq\label{eqtangent}
T_{(x,\eta',k)}\mathcal{C}_{\phr}=\{(q,p',Ak)\in \rn\times\R^{n-1}\times \g k:(R_0(q),R_0(p',0),Ak)\in T_{(R_0x,R_0(\eta',1),k)}(\Reg \mathcal{C})\}.
\eq
To compute the Hessian of $\phr$ at a point $(x_0,\eta_0',k_0)\in \mathcal{C}_{\phr}$, we fix a basis $(A_1,\dots,A_d)$ of $\g$, and use the chart $\alpha:\R^{2n-1}\times\R^d \to \R^{2n-1}\times G$  defined by
$$\alpha(x,\eta',s)=\left(x,\eta', \exp\left(\sum_{i=1}^d s_i A_i\right)k_0\right).$$
With respect to these coordinates, the Hessian of $\phr$ is given by
$$\mbox{Hess }\phr(x_0,\eta_0',k_0)=\left( \frac{\partial^2(\phr\circ \alpha)}{\partial X_i\partial X_j} (x_0,\eta_0,0) \right)_{1\leq i,j\leq 2n+d-1}$$
which is a square matrix of size $2n+d-1$.
Before entering the computations, we recall that by  (3.17) of \cite{cassanas} we have
\bq\label{roch2bis}
 \eklm{JAz,Bz}_{\R^{2n}}=0 \quad \forall z\in \Omega_0, \quad\forall A, B \in\g,
\eq
which is equivalent to
\bq\label{roch2}
\eklm{Ax,B\xi}_{\R^{n}}=\eklm{Bx,A\xi}_{\R^{n}} \quad \forall (x,\xi)\in \Omega_0,\quad \forall A, B \in\g.
\eq
Using these identities, we obtain for all $(x,\eta',k) \in \mathcal{C}_{\phr}$  that $\mbox{Hess }\phr(x,\eta',k)$ is given by 
\bqn
\left(
\begin{array}{c| c| c}
0                              &     \eklm{R_0e_i;(k^{-1}-\1)R_0e_j}  &  \eklm{R_0e_i;k^{-1}A_iR_0(\eta',1)}\\\hline
\eklm{R_0e_i;(k-\1)R_0e_j}          &            0            &  -\eklm{A_jR_0 x;R_0e_i}\\\hline
\eklm{R_0e_j;k^{-1}A_iR_0(\eta',1)} & -\eklm{A_i R_0x,R_0e_j}            &  -\eklm{A_iR_0x,A_jR_0(\eta',1)}
\end{array}\right),
\eqn
where the first diagonal block is of size $n$, the second of size $n-1$ and the third of size $d$;  each block has been characterized by specifying the entry of the $i$-th line and the $j$-th column.
 Let now $(q,p',s)\in\rn \times \R^{n-1}\times\R^d$. We set $\displaystyle{A=\sum\limits_{i=1}^d s_i A_i}$. Then $(q,p',s)\in \ker\mbox{Hess }\phr(x,\eta',k)$ if and only if
$$\left\{
\begin{array}{lr}
(\1-k)R_0(p',0)+AR_0(\eta',1)=0, & (a)\\
(k-\1)R_0(q)-AR_0x=\lambda_0 R_0(e_n),  & (b)\\
\eklm{kR_0(q);A_iR_0(\eta',1)}-\eklm{A_iR_0x;R_0(p',0)}-\eklm{A_iR_0x;AR_0(\eta',1)}=0 , \quad \forall i=1,\dots d, & (c)
\end{array}\right.$$
for some $\lambda_0\in\R$. Taking the scalar product of (b) with $R_0(\eta',1)$, we obtain $\lambda_0=0$. Using (a), we find that (c) is equivalent to
the fact that $\eklm{kR_0(q),BR_0(\eta',1)}=\eklm{kR_0(p',0),BR_0x}$ for all $B$ in $\g$. Since  $k R_0 x=R_0x$ and $kR_0(\eta',1)=R_0(\eta',1)$, we see that for all $B \in \g$, 
$$\eklm{kR_0(q),BR_0(\eta',1)}=\eklm{kR_0(p',0),BR_0x}\iff (R_0(q),R_0(p',0))\in \left[J\g (R_0x,R_0(\eta',1))\right]^\perp.$$
 But then, according to Lemma \ref{critset}, and equation (\ref{eqtangent}), we deduce that
$$\alpha(\ker\mbox{Hess }\phr(x,\eta',k))=T_{(x,\eta',k)}\mathcal{C}_{\phr},$$
which concludes the proof of the proposition.
\end{proof}

Using the preceding proposition, we are in position to apply Theorem \ref{thm:SP} to  the integral (\ref{eq:17}).
 Nevertheless, since the integrand in \eqref{eq:17} also depends on the parameter $\nu$, the derivatives of $ \tilde\sigma_{\nu\eta_n}(x,\eta)$ with respect to $x$ and $\eta'$ have to be examined carefully. Indeed, while the derivatives of $\chi_\lambda \circ \Psi_\mu^{-1}$ and $\psi_{\lambda,c}\circ \Psi_\mu^{-1}$ behave nicely in terms of $\mu$, the derivatives of $\tilde \eta_c \circ \Psi_\mu^{-1}$ with respect to $\xi$ turn out to be more delicate.

 \begin{lemma}
For all multiindices $\alpha,\beta$, there exists a constant $C>0$, which depends only on $\alpha$ and $\beta$, such that
\bqn
\sup_{(x,\eta)\in \bf{X} \times \R^{n}} |\gd ^\beta_x \gd_{\eta'}^\alpha \tilde \sigma_{\nu\eta_n }(x,\eta)|\leq C \, \max \,(1,|\nu|^{-\delta (|\beta|+|\alpha|)} ).
\eqn
\end{lemma}
\begin{proof}
With $\tilde \sigma_{\nu\eta_n}(x,\eta)=\sigma_{\nu\eta_n}(x, \eta_n(\eta_1,\dots, 1))=\tau_{\nu\eta_n}(x,(\eta_1, \dots, 1)/\nu)$, 
 $\tau_\mu=[(\eta_{\lambda,-2}^2 \chi_\lambda)^2 ( 3- 2 \eta_{\lambda,-2}^2\chi_\lambda) + f_\lambda]$  one computes
\begin{align*}
|\gd^\beta_x\gd_{\eta'}^\alpha &\tilde \sigma_{\nu \eta_n}(x,\eta)| =|\nu|^{-|\alpha|}|(\gd ^\beta_x\gd_{\eta'}^\alpha \tau_{\nu\eta_n})(x, (\eta',1)/\nu)|\\ 
&\leq C_{\alpha,\beta} \, |\nu|^{-|\alpha|}(1+|x|^2 + (|\eta'|^2 +1)/\nu^{2})^{{(\delta|\beta|-(1-\delta) |\alpha|)}/ 2}\\ 
& \leq C_{\alpha,\beta} \, |\nu|^{-\delta|\alpha|} |\nu|^{-\delta|\beta|} (\nu^2 + |\nu x|^2 +|\eta'|^2+1)^{(\delta|\beta|-(1-\delta)|\alpha|)/2} \\ 
& \leq C_{\alpha,\beta} \, |\nu|^{-\delta(|\alpha|+|\beta|)} (\nu^2 + |\nu x|^2 +|\eta'|^2+1)^{\delta|\beta|/2}.
\end{align*}
Since by \eqref{eq:suppsigma} $\sigma_\mu$ has support in a compact set independent of $\mu$, we  obtain an estimate of order $O(1)$  for large $\nu$, and  one of order $O(\nu^{-\delta (|\alpha|+|\beta|)})$ for small $\nu$.
\end{proof}
 It is interesting to note that similar bounds for $\gd_\xi^\alpha \gd_x ^\beta \sigma_\mu$ do not exist; indeed, the fact of considering only differential operators which are transversal to  $ \Reg {\mathcal{C}}$  in the variable $\xi$ turns out to be decisive.
We can now give an asymptotic expansion for $I(\mu)$.
\begin{theorem}\label{asymp}
There exists a constant $C>0$ independent of $\mu$ such that for all $\mu>0$, and  all $\delta\in(1/4,1/2)$,
\begin{gather*}
\Big |I(\mu)  - (2\pi \mu)^{\kappa}  L_0(\mu)\Big  | \leq C \mu^{\kappa+1-2\delta},
\end{gather*}
 where  $\kappa$ is given by Definition \ref{def2},  and
\bqn
L_0(\mu)=\frac{1}{(2\pi)^n} \int_{0\leq|\eta_n|\leq E_0} \int_{\mathcal{C}_{\phr}}\frac{\overline {\chi(k)} \tilde \sigma_{\mu}(x,\eta',\eta_n)}{|\det  \, \phr'' (x,\eta',k)_{|N_{(x,\eta',k)}\mathcal{C}_{\phr}} |^{1/2}} d\sigma_{\mathcal{C}_{\phr}}(x,\eta',k)
|\eta_n|^{n-\kappa-1} d\eta_n.
\eqn
\end{theorem}
\begin{proof}
In view of Proposition \ref{phaseanalysis}, we can apply Theorem \ref{thm:SP} to the integral $I_{\eta_n}(\nu)$ which was defined in \eqref{eq:17}. Consequently, for each $N\in \N$, there exists a constant $C_{N}>0$ independent of $\eta_n$ such that
\bqn
\Big | I_{\eta_n}(\nu)-(2\pi|\nu|)^\kappa \sum_{j=0}^{N-1}|\nu|^j Q_j(\eta_n)\Big |\leq C_{N} \, |\nu|^N \sup_{|\alpha|+|\beta|\leq 2N}\norm{\gd_{\eta'}^\alpha \gd_x^\beta \tilde \sigma_{\nu\eta_n}}_{\infty, {\bf X}\times \rn},
\eqn
as well as constants $\tilde C_{j}>0$ independent of $\eta_n$, such that
\bqn
| Q_j(\eta_n)| \leq \tilde C_{j} \sup _{|\alpha|+|\beta|  \leq 2j} \norm{\gd_{\eta'}^\alpha \gd_x^\beta \tilde \sigma_{\nu\eta_n}}_{\infty, {\bf X}\times \rn},
\eqn
where, in particular,
\bqn
Q_0(\eta_n)=\int_{\mathcal{C}_{\phr}}\frac{\overline {\chi(k)} \tilde \sigma_{\nu\eta_n}(x,\eta',\eta_n)}{|\det \phr'' (x,\eta',k)_{|N_{(x,\eta',k)} \mathcal{C}_{\phr}}|^{1/2}} d\sigma_{\mathcal{C}_{\phr}}(x,\eta',k).
\eqn
Now, by the previous lemma, for $|\nu|\leq 1$ one has
\bqn
\sup _{|\alpha|+|\beta| \leq 2N} \norm{\gd_{\eta'}^\alpha \gd_x^\beta \tilde \sigma_{\nu\eta_n}}_{\infty, {\bf X}\times \rn}
\leq c_1 |\nu|^{-2N\delta},
\eqn
where  $c_1$ is some constant depending only on $N$. Thus, if $|\nu|\leq 1$, we obtain
\begin{align}\label{eq:21bis}
\begin{split}
\Big |I_{\eta_n}(\nu)  - (2\pi |\nu|)^{\kappa}  Q_0(\eta_n)\Big  | &=\Big |I_{\eta_n}(\nu)-(2\pi|\nu|)^\kappa \left [\sum_{j=0}^{N-1}|\nu|^j Q_j(\eta_n) -\sum_{j=1}^{N-1}|\nu|^jQ_j(\eta_n)\right ]\Big |\qquad ~\\
&\leq C_{N} \, |\nu|^N \sup_{|\alpha+\beta|\leq 2N}\norm{\gd_{\eta'}^\alpha \gd_x^\beta \tilde \sigma_{\nu\eta_n}}_{\infty, {\bf X}\times \rn}
+\Big |(2\pi|\nu|)^\kappa \sum_{j=1}^{N-1}|\nu|^j Q_j(\eta_n)\Big |\\
&\leq c_2 |\nu|^{N(1-2\delta)}+ c_3 |\nu|^\kappa \sum_{j=1}^{N-1} |\nu|^{j(1-2\delta)}
\end{split}
\end{align}
with constants $c_i>0$. Next, let us fix $\eps>0$, and write
\bqn
I(\mu)=J_{1}(\mu)+J_2(\mu),
\eqn
where
$$\begin{array}{l}
\displaystyle{J_{1}(\mu)=\int_{\eps\leq|\eta_n|\leq E_0}I_{\eta_n}(\mu/\eta_n)|\eta_n|^{n-1}  \d\eta_n,}\\
\displaystyle{ J_{2}(\mu)=\int_{|\eta_n|\leq \eps}I_{\eta_n}(\mu/\eta_n)|\eta_n| ^{n-1} \d\eta_n.}
\end{array}$$
Since $I_{\eta_n}(\mu)$ is uniformly bounded in $\eta_n$ and $\mu$,
\bq\label{eq:21ter}
|J_{2}(\mu)|\leq {c_4} \eps^n,
\eq
where ${c_4}$ is independent of  $\eta_n$ and $\mu$. Now, according to equation (\ref{eq:21bis}), if $\eps\geq \mu$, then 
$$\Big |J_{1}(\mu)  - (2\pi \mu)^{\kappa}  \int_{\eps\leq|\eta_n|\leq E_0} Q_0(\eta_n)|\eta_n|^{n-1-\kappa} d\eta_n \Big  |$$
$$\leq C_1\sum_{j=1}^{N-1} \mu^{\kappa +j(1-2\delta)}
 \int_{\eps\leq|\eta_n|\leq E_0}|\eta_n|^{n-1-\kappa-j(1-2\delta)} d\eta_n
+C_2 \mu^{N(1-2\delta)}
 \int_{\eps\leq|\eta_n|\leq E_0}|\eta_n|^{n-1-N(1-2\delta)} d\eta_n$$
 for some constants $C_i>0$.
One easily computes that
$$\left| \int_{\eps\leq|\eta_n|\leq E_0}|\eta_n|^{n-1-\kappa-j(1-2\delta)} d\eta_n\right|\leq C_3\max\{1, \eps^{n-\kappa-j(1-2\delta)} \},$$
$$\left| \int_{\eps\leq|\eta_n|\leq E_0}|\eta_n|^{n-1-N(1-2\delta)} d\eta_n\right|\leq C_3\max\{1, \eps^{n-N(1-2\delta)} \},$$
so that if we take $\eps=\mu$, which ensures that $|\nu|\leq 1$ for  $J_{1}(\mu)$, we obtain
\bqn
\Big |J_{1}(\mu)  - (2\pi \mu)^{\kappa}  \int_{\mu\leq|\eta_n|\leq E_0} Q_0(\eta_n)|\eta_n|^{n-1-\kappa} d\eta_n \Big  |
\leq C_1 \max\{\mu^{\kappa+1-2\delta}, \mu^n\} +C_2 \max\{\mu^{N(1-2\delta)}, \mu^n\}.
\eqn
As the dimension of an orbit of $G\subset O(n)$ in $\R^n$ is at most $n-1$, one necessarily has $\kappa\leq n-1$, yielding $\mu^n=O(\mu^{\kappa+1})$ as $\mu$ goes to zero. Therefore, by choosing $N$ large enough, and taking  equation (\ref{eq:21ter}) together with 
\bqn
(2\pi \mu)^\kappa \int_{0\leq|\eta_n|\leq \mu} Q_0(\eta_n)|\eta_n|^{n-1-\kappa} d\eta_n   =
O( \mu^n)
\eqn
into account, one gets
\bqn
\Big | I(\mu)  - (2\pi \mu)^{\kappa}  \int_{0\leq|\eta_n|\leq E_0} Q_0(\eta_n)|\eta_n|^{n-1-\kappa} d\eta_n \Big  |
\leq C\mu^{\kappa+1-2\delta}.
\eqn
The proof of the theorem is now complete.
\end{proof}
\begin{remark}
Note that the strict transform of the critical set $\mathcal{C}$ of $\psi$ is locally given by
\bqn\tilde{\mathcal{C}}^{\mathrm{st}}=\{(x,\eta,k)\in\rnn\times G\,:\, (R_0x,R_0(\eta',1),k)\in \Reg {\mathcal{C}}\}\simeq \mathcal{C}_{\phr}\times \R.
\eqn
The first coefficient in the expansion of Theorem \ref{asymp} can therefore also be expressed as
\bq\label{coco}
L_0(\mu)=\frac 1 {(2\pi)^{n}}
\int_{\tilde{\mathcal{C}}^{\mathrm{st}}}\frac{\overline {\chi(k)}  \tilde \sigma_{\mu}(x,\eta)|\eta_n|^{n-\kappa-1}}{|\det  \, \phr'' (x,\eta',k)_{|N_{(x,\eta',k)}\mathcal{C}_{wk}} |^{1/2}} d\sigma_{\tilde {\mathcal{C}}}(x,\eta,k).
\eq
\end{remark}

\section{Computation of the leading term}

In this section, we shall address the question of computing the leading coefficient $L_0(\mu)$ in the expansion of $I(\mu)$.
The main result of this section is the following
\begin{proposition}
\label{computhess}
One has
\bq\label{eq:FQ}
L_0(\mu)=\frac 1 {(2\pi)^n} [{\rho_\chi}_{|H_0}:\1]\,   \int_{\Reg \Omega_0} \sigma_\mu(z) \frac{d\sigma_{\Reg \Omega_0}(z)}{{\vol }\mathcal{O}_z},
\eq
where $d\sigma_{\Reg \Omega_0}$ is the Riemannian measure on $\Reg \Omega_0$, and $\vol \mathcal{O}_z$ denotes the Riemannian volume of the $G$-orbit of $z$. In particular, the integral on the right hand side  of \eqref{eq:FQ} is convergent.
\end{proposition}
Note that $\Reg \Omega_0$ is  not compact; nevertheless, the existence of the integral in \eqref{eq:FQ} will be deduced on basis of the partial desingularization of $\mathcal{C}$ accomplished in the previous section.
Let us start proving Proposition \ref{computhess}, and introduce first certain cut-off functions
 for $\Sing \Omega_0$.
\begin{definition}
 Let $K$ be compact subset in $\R^{2n}$ as in \eqref{eq:suppsigma}, $\epsilon >0$, and denote by $v_\epsilon$ the characteristic function of the set
 \bqn
 (\Sing \Omega_0 \cap K)_{2\epsilon}=\mklm{z \in \R^{2n}: |z-z'| <2 \epsilon \text{ for some } z' \in \Sing \Omega_0\cap K}.
 \eqn
 Consider further the unit ball $B_1$ in $\R^{2n}$, and a function $\iota \in \CT(B_1)$ with $\int \iota dz=1$, and set $\iota_\epsilon(z) =\epsilon^{-2n} \iota(z/\epsilon)$. Clearly  $\int \iota_\epsilon dz =1$, $\supp \iota_\epsilon \subset B_\epsilon$, and we define
 \bqn
 u_\epsilon=v_\epsilon \ast \iota_\epsilon.
 \eqn
\end{definition}
One can then show that $u_\epsilon \in \CT((\Sing \Omega_0 \cap K)_{3\epsilon})$, and $u_\epsilon =1$ on $(\Sing \Omega_0\cap K)_\epsilon$, together with
\bqn
|\gd^\alpha_z u _\epsilon|\leq C_\alpha \epsilon^{-|\alpha|},
\eqn
where $C_\alpha$ is a constant which depends only on $\alpha$ and $n$, see H\"{o}rmander \cite{hoermanderI},  Theorem 1.4.1.\\\

Next,  we shall prove
\begin{lemma}\label{lemlim}
Let $\alpha \in \CT(\R^{2n})$. Then the  limit 
$$\lim_{\eps \to 0}  \int_{\Reg {\mathcal{C}}}\frac{\overline {\chi(k)}  [\alpha(1-u_\eps)] (z)}{|\det  \, \psi'' (z,k)_{|N_{(z,k)}\Reg {\mathcal{C}}} |^{1/2}} d\sigma_{\Reg {\mathcal{C}}}(z,k)$$
exists and is finite.
In particular, one has
\bq\label{eq:MC}
L_0(\mu)
=\frac 1 {(2\pi)^{n}} \lim_{\eps \to 0}  \int_{\Reg {\mathcal{C}}}\frac{\overline {\chi(k)}  [\sigma_{\mu}(1-u_\eps)] (z)}{|\det  \, \psi'' (z,k)_{|N_{(z,k)}\Reg {\mathcal{C}}} |^{1/2}} d\sigma_{\Reg {\mathcal{C}}}(z,k).
\eq
where $d\sigma_{\Reg {\mathcal{C}}}$ is the Riemannian measure on $\Reg {\mathcal{C}}$.
\end{lemma}
\begin{proof}
With $u_\eps$ as in the previous definition, let us define
\bqn
I_\eps(\mu)=\int _G\int_{\bf X} \int_{\R^n}  e^{\frac{i}{\mu} \psi(x,\xi,k) }  \overline{\chi(k)}[\alpha(1-u_\eps)](x,\xi) dx \, \dbar \xi dk.\eqn
Since $(x,\xi,k) \in \Sing {\mathcal{C}}$ implies $ (x,\xi) \in \Sing \Omega_0$, a direct application of the generalized theorem of the stationary phase for fixed $\eps >0$ gives
\bq
\label{eq:asympt}
| I_\eps(\mu)- (2\pi \mu)^\kappa L_0(\mu,\eps) | \leq C_\eps \mu^{\kappa+1 -2\delta}
\eq
for some $\delta \in [0,1/2)$,
where $C_\eps>0$ is a constant depending only on $\eps$, and
\bqn
L_0(\mu,\eps)=\frac 1 {(2\pi)^{n}} \int_{\Reg {\mathcal{C}}}\frac{\overline {\chi(k)}  [\alpha(1-u_\eps)] (z)}{|\det  \, \psi'' (z,k)_{|N_{(z,k)}\Reg {\mathcal{C}}} |^{1/2}} d\sigma_{\Reg {\mathcal{C}}}(z,k).
\eqn
If $\alpha$ is independent of $\mu$, on has $\delta=0$. For $\alpha = \sigma_\mu$, the stationary phase theorem has to be applied on $G\times {\bf{X}} \times S^{n-1}$, and $\delta \in (1/4,1/2)$.
On the other hand, applying Theorem \ref{asymp} to $I_\eps(\mu)$ instead of $I(\mu)$, we obtain again an asymptotic expansion of the form \eqref{eq:asympt} for $I_\eps(\mu)$, where now, according to \eqref{coco},
the first coefficient is given by
\bqn
L_0(\mu,\epsilon)=\frac 1 {(2\pi)^{n}}
\int_{\tilde{\mathcal{C}}^{\mathrm{st}}}\frac{\overline {\chi(k)}  [\alpha(1-u_\eps) \circ \zeta_M\circ \tilde \phi_n] (x,\eta)|\eta_n|^{n-\kappa-1}}{|\det  \, \phr'' (x,\eta',k)_{|N_{(x,\eta',k)}\mathcal{C}_{wk}} |^{1/2}} d\sigma_{\tilde {\mathcal{C}}}(x,\eta,k).
\
\eqn
Since the first term in the asymptotic expansion \eqref{eq:asympt} is uniquely determined, the two expressions for $L_0(\mu,\eps)$ must be identical. The statement of the lemma now follows by  the Lebesgue theorem on bounded convergence, by which, in particular,
\bqn
\lim _{\eps \to 0} \frac 1 {(2\pi)^{n}}
\int_{\tilde{\mathcal{C}}^{\mathrm{st}}}\frac{\overline {\chi(k)}  [\sigma_{\mu}(1-u_\eps) \circ \zeta_M\circ \tilde \phi_n] (x,\eta)|\eta_n|^{n-\kappa-1}}{|\det  \, \phr'' (x,\eta',k)_{|N_{(x,\eta',k)}\mathcal{C}_{wk}} |^{1/2}} d\sigma_{\tilde {\mathcal{C}}}(x,\eta,k) = L_0(\mu).
\eqn
\end{proof}
\begin{remark}\label{remark4}
Note that existence of the limit in  \eqref{eq:MC} has been established by  partially resolving the  singularities of the critical set $\mathcal{C}$, the corresponding limit being given by the absolutely convergent integral \eqref{coco}.
\end{remark}
\begin{lemma}\label{formulaint}
Let $\alpha$ be a smooth, compactly supported function on $\Reg \Omega_0$. Then
\bqn
\int_{\Reg {\mathcal{C}}}\frac{\overline {\chi(k)}  \alpha(z)}{|\det  \, \psi'' (z,k)_{|N_{(z,k)}\Reg {\mathcal{C}}} |^{1/2}} d\sigma_{\Reg {\mathcal{C}}}(z,k)
=[{\rho_\chi}_{|H_0}:\1]\int_{\Reg \Omega_0} \alpha (z) \frac{d\sigma_{\Reg \Omega_0}(z)}{\mbox{Vol }\mathcal{O}_z}.
\eqn
\end{lemma}
\begin{proof}
The main difficulty consists in computing the determinant of the transversal Hessian, which will be accomplished by  recuring  to previous  computations done in \cite{cassanas}. Thus, let $(z,k)$ be a fixed point in $\Reg {\mathcal{C}}$, and choose an appropriate  basis  $(A_1,\dots,A_d)$ for $\g$  as follows. If  $\kappa$ denotes the dimension of $\mathcal{O}_z$,  let
$$(A_1,\dots,A_\kappa) \mbox { be an orthonormal basis of } (T_e G_z)^\perp,$$
$$(A_{\kappa+1},\dots,A_d) \mbox { be an orthonormal basis of } T_e G_z,$$
where orthogonality is defined with respect to the scalar product
$$\eklm{\eklm{A,B}}=\tr(^tAB)$$
for arbitrary linear maps $A$ and $B$ in $\rn$.  From \cite{cassanas} we recall that
$$\det\left(\frac{ \psi''(z,k)_{|_{ \mathcal{N}_{(z,k)\Reg {\mathcal{C}}}}} }{ i }\right)=\det\left(\frac{\mathcal{A}_{|_{\F^\perp}}}{i}\right),$$
where $\mathcal{A}=\mbox{Hess } \psi(z,k)$ denotes the Hessian of $\psi$ with respect to the coordinates $(z,s) \to (z,\exp (\sum_{i=1}^d s_i A_i)k)$, and
\bq
\F=\left\{(\alpha, s)\in\rnn\times\R^{d}:\quad (k-\1)\alpha+\sum_{i=1}^d s_i A_iz=0 \right\}.
\eq
Next, let $(B_1, \dots,B_\kappa)$ be in $\g$ such that $(B_1 z, \dots,B_\kappa z)$ is an orthonormal basis of $\g z$. For $j=1,\dots,\kappa$, we define
\bq
\eps_j=(JB_j z,0), \qquad \eps_j'=((\kmu-\1)B_j z,\eklm{A_i z,B_j z},0), \qquad (i=1,\dots,\kappa).
\eq
Then $(\eps,\eps')$ constitutes a basis of $\F^\perp$, see \cite{cassanas}, Lemma 3.3.
In what follows, we shall compute  $\mathcal{A}_{|_{\F^\perp}}$ in this basis. Writing
$\alpha_j=(kB_jx,B_j \xi)$
we find
\bq\label{roch3}
\A \eps_j=((\kmu-\1)(\1-\Pi_{\g z})\alpha_j,0) + \sum_{r=1}^\kappa \eklm{\alpha_j,B_r z}\eps_j',
\eq
where $\Pi_{\g z}$ is the orthogonal projection onto the space $\g z$ in $\rnn$. We state now certain relations that will be crucial for the rest of the computation.
For all $(z,k)\in \Reg \mathcal{C}$, we have
\bq\label{roch3bis}
[k,\Pi_{\g z}]=0, \qquad [J,k]=0.
\eq
\bq\label{roch3ter}
\mbox{rank }[(k-\1)(\1-\Pi_{\g z})]\subset J\g z.
\eq
The first equality follows easily from the relations $\kmu \g k=\g$ and $kz=z$, while the second simply says that $k$ is symplectic as a Hamiltonian action in $\R^{2n}$. In order to establish \eqref{roch3ter}, we differentiate  the identity
$$\pi(kz)=\pi(z)$$
with respect to $z\in \Omega_0$, and  obtain  $(k-\1)\alpha\in \ker d_z\pi=\g z$ for all $\alpha$ in $T_z\Omega_0$, where $\pi$ denotes the canonical projection of $\R^{2n}_{(H_0)}$ onto the quotient by $G$.  The inclusion  \eqref{roch3ter} now follows by using \eqref{roch1}.
Coming back to \eqref{roch3}, we get
$$\A \eps_j=\sum_{r=1}^\kappa -\eklm{J(\kmu-\1)\alpha_j,B_r z}\eps_r +\sum_{r=1}^\kappa \eklm{\alpha_j,B_r z}\eps_r'.$$
Using (\ref{roch2bis}), and the fact that $(B_1 z, \dots,B_\kappa z)$ is orthonormal, we obtain
\bq\label{roch4}
\A \eps_j=\sum_{r=1}^\kappa \eklm{(\1-k)(\1-\kmu)B_jx,B_r \xi}\eps_r +\sum_{r=1}^\kappa [\eklm{(k-\1)B_jx,B_r x}-\delta_{jr}]\eps_r',
\eq
where $\delta_{jr}$ is the Kronecker symbol.
In the same way we obtain
$$\A\eps_j'=\sum_{r=1}^\kappa -\eklm{J(\kmu-\1)\beta_j +\oh (\kmu+I) C_j z,B_r z}\eps_r +\sum_{r=1}^\kappa \eklm{\beta_j,B_r z}\eps_r',$$
where
$$C_j=\sum_{r=1}^\kappa \eklm{A_r z,B_j z} A_r, \qquad \beta_j=(\kmu-\1)(-kBj \xi, B_j x)-\oh(C_j \xi,C_j x).$$
Let now $f:\g z\to\g z$ be defined by
\bq
\label{eq:44b}
f(\tilde z)=\sum_{r=1}^\kappa \eklm{A_r z,\tilde z} A_rz, \qquad \forall \tilde z \in \g z,
\eq
and let
\bqn
\Lambda=\left((k-\1)(\kmu-\1)+f\right)_{|_{\g z}}
\eqn
be the restriction of  the map $(k-\1)(\kmu-\1)+f$ to $\g z$.
Note that $\Lambda$ plays a crucial part in the computations of \cite{cassanas}. Using again (\ref{roch2bis}), one easily gets
$$\A\eps_j'=\sum_{r=1}^\kappa \eklm{\left(
\begin{array}{cc}
\kmu & 0\\
0 & \1_n
\end{array}\right)
\Lambda B_j z,B_r z}\eps_r -\sum_{r=1}^\kappa \eklm{\left(
\begin{array}{cc}
0 & \1_n\\
0 & 0
\end{array}\right)\Lambda B_j,B_r z}\eps_r',$$
where the matrices have an obvious meaning. Together with (\ref{roch4}), the last equation implies that the matrix of $\A$ in the basis $(\eps,\eps')$ is given by
\bq\label{roch5}
\left(
\begin{array}{c|c}
\eklm{(\1-k)(\1-\kmu)B_jx,B_i \xi}    &
\eklm{\left(\begin{array}{cc}
\kmu & 0\\
0 & \1_n
\end{array}\right)
\Lambda B_j z,B_i z}\\\hline
\eklm{(k-\1)B_jx,B_i x}-\delta_{ij}  &
-\eklm{\left(
\begin{array}{cc}
0 & \1_n\\
0 & 0
\end{array}\right)\Lambda B_j,B_i z}
\end{array}\right).
\eq
Let $\Lambda_0$ be the matrix of $\Lambda$ in the basis $(B_1 z, \dots,B_\kappa z)$.
Then (\ref{roch5}) is equal to
$$\left(
\begin{array}{c|c}
\eklm{(\1-k)(\1-\kmu)B_jx,B_i \xi}    &
\eklm{\left(\begin{array}{cc}
\kmu & 0\\
0 & \1_n
\end{array}\right)
B_j z,B_i z}\\\hline
\eklm{(k-\1)B_jx,B_i x}-\delta_{ij}  &
-\eklm{\left(
\begin{array}{cc}
0 & \1_n\\
0 & 0
\end{array}\right) B_j z,B_i z}
\end{array}\right).
\left(\begin{array}{cc}
      \1_\kappa & 0\\
      0        & \Lambda_0
      \end{array}
\right).$$
Multiplying by $i$, and shifting the two columns, we obtain
\bqn
\det\left(\frac{ \psi''(z,k)_{|_{ \mathcal{N}_{(z,k)\mathcal{C}_0}}} }{ i }\right)=
\det(\Lambda)\cdot \D,
\eqn
where
\bq\label{roch6bis}
\D=\det\left(
\begin{array}{c|c}
\eklm{(\kmu-\1)B_jx,B_i x}+\delta_{ij}    & \eklm{(k-\1)(\kmu-\1)B_j\xi ,B_i x} \\ \hline
- \eklm{B_j\xi,B_i x}                    & \eklm{(k-\1)B_jx,B_i x}+\delta_{ij}
\end{array}\right).
\eq
We are going to show that $\D=1$. For this, we introduce the notation$$U=\left(
\begin{array}{c|c|c}
B_1 x & \dots & B_\kappa x
\end{array}\right),
\qquad
V=\left(
\begin{array}{c|c|c}
B_1 \xi & \dots & B_\kappa \xi
\end{array}\right),$$
 where  $B_j x$ is taken as a column vector in the canonical basis of $\rn$.
$U$ and $V$ are therefore matrices of size $n\times \kappa$. 
\begin{lemma}\label{UV} For all $k \in G$ we have
\begin{enumerate}
\item[(a)] $^tUU+^tVV=\1_\kappa$;
\item[(b)] $^tUV=^tVU$;
\item[(c)] $k$ commutes with $U ^tU$, $V ^tV$, $U ^tV$, and $V^t U$;
\item[(d)] $(k-\1)U^tV=(k-\1)V^tU$;
\item[(e)] $(k-\1)(U^tU+V^tV)=k-\1$.
\end{enumerate}
\end{lemma}
{\it Proof.} (a) says that $(B_1 z, \dots,B_n z)$ is  orthonormal. (b) comes from (\ref{roch2}). Next, let us denote by $X$ the matrix
$X=\left(
\begin{array}{c|c|c}
B_1 z & \dots & B_\kappa z
\end{array}\right).$ Then $X ^tX$ is the matrix of $\Pi_{\g z}$ in the canonical basis of $\rnn$. Moreover,
$$X ^tX=\left(
\begin{array}{c|c}
U ^tU & U ^tV\\\hline
V ^tU & V ^tV
\end{array}\right).$$
Therefore the property $[\Pi_{\g z},k]=0$, see (\ref{roch3bis}), is equivalent to (c).
The two last properties are more subtile. One has to note that (\ref{roch3bis}) is equivalent to
$$\Pi_{\g z} (k-\1)J(\1-\Pi_{\g z})=(k-\1)J(\1-\Pi_{\g z}).$$
By expressing  this in terms of matrices, one easily obtains (d) and (e).\qed\\\\
Coming back to the proof of Lemma \ref{formulaint}, we rewrite equation (\ref{roch6bis}) as
$$\D=\det\left(
\begin{array}{c|c}
^tU(k-\1)U+\1_\kappa    & ^tV(\kmu-\1)(k-\1)U \\ \hline
- ^tVU                 & ^tU(\kmu-\1)U+\1_\kappa
\end{array}\right)=\det\left(\begin{array}{cc}
a & b\\
c & d
\end{array}\right),$$
where we replaced $k^{-1}$ by $k$.
We claim that the blocks $c$ and $d$ commute. Indeed,
$$cd=-^tVU ^tU(\kmu-\1)U-^tV U,$$
$$dc=-^tU(\kmu-\1)U ^tVU- ^tVU=-^tU(\kmu-\1)V ^tUU- ^tVU,$$
by (d) of Lemma \ref{UV}. By (c) of Lemma \ref{UV}, $(\kmu-\1)$ commutes with $V ^tU$, and since $^tUV=^tVU$, by  (b), we get  $[c,d]=0$.
Therefore,  $\D=det(ad-bc)$. Using (a) and (d) of Lemma \ref{UV}, it is then a straightforward computation to show that in fact, $ad-bd=\1_\kappa$, yielding $\D=1$.
We have thus shown the equality
\bqn
\det\left(\frac{ \psi''(z,k)_{|_{ \mathcal{N}_{(z,k)\mathcal{C}_0}}} }{ i }\right)=\det \left((k-\1)(\kmu-\1)_{|_{\g z}}+f\right),
\eqn
where the map $f:\g z \rightarrow \g z$ was defined in \eqref{eq:44b}.
The rest of the proof of Lemma \ref{formulaint} now follows by the argument given in \cite{cassanas}, Section 3.3.2.
\end{proof}
To finish  proving Proposition \ref{computhess}, we note that, as a consequence of Lemmata \ref{lemlim} and \ref{formulaint},   the limit
$$\lim_{\eps \to 0}  \int_{\Reg \Omega_0} [\alpha(1-u_\eps)](z) \frac{d\sigma_{\Reg \Omega_0}(z)}{\mbox{Vol }\mathcal{O}_z}$$
exists for any $\alpha\in\CT(\rnn)$ and is finite. Assume now  that  $\alpha$ is non-negative. Since $|u_\eps|\leq 1$, the Lemma of Fatou implies 
$$\int_{\Reg \Omega_0} \lim_{\eps \to 0} [\alpha(1-u_\eps)](z) \frac{d\sigma_{\Reg \Omega_0}(z)}{\mbox{Vol }\mathcal{O}_z}
\leq \lim_{\eps \to 0}  \int_{\Reg \Omega_0} [\alpha(1-u_\eps)](z) \frac{d\sigma_{\Reg \Omega_0}(z)}{\mbox{Vol }\mathcal{O}_z} <\infty,$$
which means that
\bq\label{eq:fin}
\int_{\Reg \Omega_0} \alpha(z) \frac{d\sigma_{\Reg \Omega_0}(z)}{\mbox{Vol }\mathcal{O}_z} <\infty  \quad \forall \alpha\in\CT(\rnn,\R_+).
\eq
In particular, if $\alpha$ is taken to be equal  $1$ on the compact set $K$ specified in \eqref{eq:suppsigma},  we obtain
\bq\label{intconv}
\int_{\Reg \Omega_0} |\sigma_\mu(z)| \frac{d\sigma_{\Reg \Omega_0}(z)}{\mbox{Vol }\mathcal{O}_z} 
\leq C \int_{\Reg \Omega_0} \alpha(z) \frac{d\sigma_{\Reg \Omega_0}(z)}{\mbox{Vol }\mathcal{O}_z}
<\infty
\eq
for some $C>0$. Now,  by   Lemmata \ref{lemlim} and \ref{formulaint},
\bq\label{eq:eureka}
L_0(\mu)=\frac 1 {(2\pi)^n}  [{\rho_\chi}_{|H_0}:\1]  \lim_{\eps \to 0}
   \int_{\Reg \Omega_0} [\sigma_\mu(1-u_\eps)](z) \frac{d\sigma_{\Reg \Omega_0}(z)}{\mbox{Vol }\mathcal{O}_z} .
\eq
Since  (\ref{intconv}) implies that the integrand in \eqref{eq:eureka} has an integrable majorant for arbitrary $\eps$, we can apply the Lebesgue Theorem of bounded convergence to obtain 
\bqn
L_0(\mu)= \frac 1 {(2\pi)^n}   [{\rho_\chi}_{|H_0}:\1] \int_{\Reg \Omega_0} \sigma_\mu(z) \frac{d\sigma_{\Reg \Omega_0}(z)}{\mbox{Vol }\mathcal{O}_z}.
\eqn
This completes the proof of Proposition \ref{computhess}.\\\\

 So far we have shown that $\tr P_\chi\E_\lambda= d_\chi \lambda^{n/2m} I(\lambda^{-1/2m})+O(1)$, where
\bq
\label{eq:36a}
I(\mu)=\frac {\mu^{\kappa}} {(2\pi)^{n-\kappa}} \, [{\rho_\chi}_{|H_0}:\1]
   \int_{\Reg \Omega_0} \sigma_\mu(z) \frac{d\sigma_{\Reg \Omega_0}(z)}{\mbox{Vol }\mathcal{O}_z} \, +O(\mu^{\kappa+1-2\delta}),
\eq
$\delta \in (1/4,1/2)$, and $\sigma_\mu=[( \eta_{\lambda,-2}^2 \chi_\lambda)^2 ( 3- 2 \eta_{\lambda,-2}^2\chi_\lambda) + f_\lambda]\circ \Psi_\mu^{-1}$ with $\lambda=\mu^{-2m}$. In particular, the last integral exists, and is finite, so that in order to finish the computation of the leading term in the asymptotic expansion for $\tr P_\chi\E_\lambda$, we are left with the task of examining the latter integral.  To characterize the support of $\sigma_\mu$,  let us introduce  the  sets  \begin{align*}
 \label{eq:18}
 \begin{split}
 W_\lambda&=\mklm{(x,\xi) \in {\bf{X}} \times \R^n: a_\lambda < 0}, \\
 A_{c,\lambda}&= \mklm{(x,\xi) \in {\bf{X}} \times \R^n: a_\lambda < c( h^{\delta-\omega}+d)},  \qquad B_{c,\lambda}= {\bf{X}} \times \R^n - A_{c,\lambda},\\
 D_c&=(\gd {\bf{X}} \times \R^n) (c,h^{-2 \delta}g), \\
 F_\lambda&=\mklm{(x,\xi)\in {\bf{X}} \times \R^n: \chi_\lambda=0 \quad \text{or}  \quad \eta_{\lambda,-2} =0 \quad \text{or} \quad \chi_\lambda=\eta_{\lambda,-2} =1 },\\
 {\mathcal{RV}}_{c,\lambda}&=\mklm{(x,\xi) \in {\bf{X}} \times \R^n: | a_\lambda|< c(h^{\delta-\omega} + d)}\cup \mklm {(x,\xi) \in D_c: x \in {\bf{X}}, \, a_\lambda < c (h^{\delta -\omega} +d)}.
  \end{split}
  \end{align*}
 Note that $D_c=\mklm{ (x,\xi)\in \R^{2n}: \dist (x,\gd {\bf{X}}) <  \sqrt c\big (1 + |x|^2+|\xi|^2\big )^{-\delta/2}}$, since for
 \bdm
 h^{-2\delta}(x,\xi)g_{(x,\xi)} ( x-y, \xi - \eta)= (1+|x|^2 + |\xi|^2)^\delta \Big [ \frac { |\xi-\eta|^2}{1 +  | x|^2 +|\xi|^2} + |  x -y|^2 \Big ] < c
 \edm
to hold  for some $(y,\eta) \in \gd {\bf{X}} \times \R^n $,  it is  necessary and sufficient  that  $|x-y|^2(1 +|x|^2+|\xi|^2)^\delta<c$ is satisfied for some $y \in \gd {\bf{X}}$.
\begin{lemma}\label{lemma:11}
For sufficiently large $c>0$ one has
\begin{enumerate}
\item[(i)] $\supp f_\lambda  \subset {\mathcal{RV}}_{c,\lambda}\subset A_{c,\lambda}$;
\item[(ii)] $\supp (\eta_{\lambda,-2}^2 \chi_\lambda )^2( 3- 2 \eta_{\lambda,-2}^2\chi_\lambda) \subset A_{c,\lambda}$;
\item[(iii)]  $(\eta_{\lambda,-2}^2 \chi_\lambda )^2( 3- 2 \eta_{\lambda,-2}^2\chi_\lambda)=1$  on $W_\lambda \cap \complement _{{\bf{X}}\times \R^{n}}  {\mathcal{RV}}_{c,\lambda}$.
\end{enumerate}
\end{lemma}
\begin{proof}
As already explained in Part I, Equation (51), the support of $f_\lambda$ is contained in $\complement_{{\bf{X}} \times \R^n}  F_\lambda$, the complement of $F_\lambda$ in ${\bf{X}} \times \R^n$. Furthermore, for sufficiently large $c>0$, the set $\complement_{{\bf{X}} \times \R^n} F_\lambda$ is contained in ${\mathcal{RV}}_{c,\lambda}$, which is a consequence of the inclusions
\bq
\label{eq:29}
\complement_{{\bf{X}} \times \R^n} F_\lambda \subset A_{c,\lambda} \cap \complement _{{\bf{X}}\times \R^n} E_\lambda \subset {\mathcal{RV}}_{c,\lambda},
\eq
where $E_\lambda= \mklm{ (x,\xi) \in {\bf{X}} \times \R^n:  (x,\xi) \not\in D_4, a_\lambda < -4h^{\delta-\omega} - 8 C_0 d }$, see  Part I, Lemma 16.  Next,  we note that $(\eta_{\lambda,-2}^2 \chi_\lambda )^2( 3- 2 \eta_{\lambda,-2}^2\chi_\lambda)(x,\xi)$ must be equal  $1$ on $W_\lambda \cap \complement _{{\bf{X}}\times \R^{n}}  {\mathcal{RV}}_{c,\lambda}$, since according to \eqref{eq:29} we have the inclusion $\complement_{{\bf{X}} \times \R^n}{\mathcal{RV}}_{c,\lambda} \subset B_{c,\lambda} \cup E_\lambda$, and hence $W_\lambda \cap \complement _{{\bf{X}}\times \R^{n}}  {\mathcal{RV}}_{c,\lambda} \subset E_\lambda \subset \mklm {(x,\xi) \in {\bf{X}} \times \R^n: \chi_\lambda =\eta_{\lambda,-2} =1}$, due to the fact that $W_\lambda \cap B_{c,\lambda}=\emptyset$. Furthermore, $(\eta_{\lambda,-2}^2 \chi_\lambda)^2 ( 3- 2 \eta_{\lambda,-2}^2 \chi_\lambda)(x,\xi)$ vanishes on $B_{c,\lambda}$, since for large $c$, $(x,\xi) \in B_{c,\lambda}$ implies $(x,\xi) \not \in \M_\lambda(1, h^{-2\delta}g)$, by the proof of the previous lemma.
\end{proof}
Consequently, by introducing the sets
 \begin{align*}
 \begin{split}
 \widetilde{W}_\mu=&\Psi_{\mu}(W_{\mu^{-2m}})=\mklm{(x,\xi) \in {\bf{X}} \times \R^n: b_{\mu^{-2m}} < 0}, \\
 \widetilde{A}_{c,\mu}=&\Psi_{\mu}(A_{c,\mu^{-2m}})= \mklm{(x,\xi) \in {\bf{X}} \times \R^n: b_{\mu^{-2m}} < c( h^{\delta-\omega}+d)\circ \Psi_\mu^{-1})}, \\
  \widetilde{B}_{c,\mu}=& {\bf{X}} \times \R^n - \widetilde{A}_{c,\mu},\\
 \widetilde{\mathcal{RV}}_{c,\mu}=&\Psi_{\mu}({\mathcal{RV}}_{c,\mu^{-2m}})=\mklm{(x,\xi) \in {\bf{X}} \times \R^n: |b_{\mu^{-2m}}|< c(h^{\delta-\omega} + d)\circ \Psi_\mu^{-1}}
 \\&\cup \mklm {(x,\xi) \in {\bf{X}} \times \rn: (x,\xi/\mu) \in D_c, \, b_{\mu^{-2m}} < c (h^{\delta -\omega} +d)\circ \Psi_\mu^{-1}},
  \end{split}
  \end{align*}
one sees that for all $\mu \in \R^+_\ast$
\bq
\label{eq:suppsigma}
\supp \sigma_\mu \subset \tilde A_{c,\mu} \subset K
\eq
for some sufficiently large $c>0$, and some suitable compact subset $K\subset \R^{2n}$. We proceed now to split the integral in \eqref{eq:36a} into the three integrals
\begin{align}
\label{eq:37c}
\begin{split}
\int_{\Reg \Omega_0\cap\widetilde{W}_\lambda} &\frac {d\sigma_{\Reg \Omega_0}(z)}{\vol \mathcal{O}_{z}}
-\int_{\Reg \Omega_0\cap\widetilde{W}_\lambda \cap\widetilde{\mathcal{RV}}_{c,\mu}} \frac {d\sigma_{\Reg \Omega_0}(z)}{\vol \mathcal{O}_{z}}  \\ &+
 \int_{\Reg \Omega_{0}\cap\widetilde{\mathcal{RV}}_{c,\mu}}
 \sigma_\mu(z)   \frac{\d\sigma_{\Reg \Omega_0}(z)}  {\vol \mathcal{O}_{z}},
\end{split}
\end{align}
where we made use of the fact that, since $W_\lambda, {\mathcal{RV}}_{c,\lambda}$ are contained in $ A_{c,\lambda}$, and  $\complement _{A_{c,\lambda}} {\mathcal{RV}}_{c,\lambda} \subset W_\lambda$, one has  $A_{c,\lambda} -W_\lambda \cap \complement _{{\bf{X}}\times \R^{n}}  {\mathcal{RV}}_{c,\lambda}= {\mathcal{RV}}_{c,\lambda} $.
 The next lemma will show that the main contribution to $L_0(\mu)$ is actually given by the first integral  in \eqref{eq:37c}, provided that we make the following  
\begin{assumption}\label{assump2}
There exists a constant $c>0$ such that for sufficiently small  $\rho>0$, $\vol (\gd {\bf{X}})_\rho  \leq c\rho$. Furthermore, $0 \not\in \gd {\bf{X}}$.
\end{assumption}
\begin{lemma}
\label{lemma:10}
Put 
\begin{align*}
 \widetilde{\mathcal{RV}}_{c,\mu}^{(1)}=&\mklm{(x,\xi) \in {\bf{X}} \times \R^n: | b|< c(h^{\delta-\omega} + d)\circ \Psi_\mu^{-1}}, \\
  \widetilde{\mathcal{RV}}_{c,\mu}^{(2)}=&\mklm {(x,\xi) \in {\bf{X}} \times \R^n: (x,\xi/\mu) \in D_c, \, b < c (h^{\delta -\omega} +d)\circ \Psi_\mu^{-1}},
\end{align*}
so that $ \widetilde{\mathcal{RV}}_{c,\mu}= \widetilde{\mathcal{RV}}_{c,\mu}^{(1)}\cup  \widetilde{\mathcal{RV}}_{c,\mu}^{(2)}$. Then, as $\mu \to 0$,
\begin{align*}
\int_{ \Reg \Omega_{0} \cap\widetilde{\mathcal{RV}}_{c,\mu}^{(1)}}  \frac{d\sigma_{\Reg \Omega_0}(z)}{\vol \Ocal_z}&=O(\mu^{2\delta -\frac 1 2}), \\
\int_{ \Reg \Omega_{0} \cap\widetilde{\mathcal{RV}}_{c,\mu}^{(2)}}  \frac{d\sigma_{\Reg \Omega_0}(z)}{\vol \Ocal_z}&=O(\mu^{\frac \delta {1+\delta}}), 
\end{align*}
for arbitrary $\delta \in (1/4,1/2)$.
\end{lemma}
\begin{proof} Let $\1_A$ denote the characteristic function of the set $A$.
As already noted, $\Omega_0$ is  homogeneous in $x$ and $\xi$, meaning that $(x,\xi) \in \Omega_0$ implies $(sx,t\xi) \in \Omega_0$ for all $s,t \in \R$. Furthermore, by Lemma \ref{lemma:parametr}, $\{(x,\xi) \in \Reg \Omega_0: \xi \in \Sing \rn\}$  is a subset of measure zero in $\Reg \Omega_0$. Consequently, we can  parametrize $\Reg \Omega_0$ up to a set of measure zero as follows. Take $z=(x,\xi) \in \Omega_0$, $\xi \in \rn_{(H_0)}$, and let $\xi =s\eta$,  $x=r\vartheta$ be polar coordinates in $\R^{n}$, and $N_\xi \Ocal_\xi$, respectively,  where $r,s >0$, and $\eta \in S^{n-1}$, $\vartheta \in S^{n-\kappa-1}$. In this coordinates one computes then
\begin{gather}
\label{eq:21a}
\begin{split}
\int_{ \Reg \Omega_{0} \cap\widetilde{\mathcal{RV}}_{c,\mu}}  \frac{d\sigma_{\Reg \Omega_0}(z)}{\vol \Ocal_z}=\int_{\rn_{(H_0)}} \Big ( \int_{N_\xi\Ocal_\xi} \1_{ \widetilde{\mathcal{RV}}_{c,\mu}}(x,\xi) \frac{d\sigma_{N_\xi\Ocal_\xi}(x)}{\vol \Ocal_{(x,\xi)}} \Big ) d\xi
\\ =\int_{0}^\infty \int_{S^{n-1}_{(H_0)}} \Big (
 \int_0^\infty  \int_{N^1_{s\eta} \Ocal_{s\eta}}  \1_{ \widetilde{\mathcal{RV}}_{c,\mu}}(r\vartheta,s\eta)  s^{n-1} r^{n-\kappa-1} \frac{ dr \, d\vartheta }{\vol \Ocal_{(r\vartheta,s\eta)}} \Big )  \,ds \,d\eta,
\end{split}
\end{gather}
since
\bqn
\sqrt{\det g_{|\Reg \Omega_0}(r,s,\vartheta,\eta)}=s^{n-1} r^{n-\kappa-1} d\vartheta \, d\eta,
\eqn
where $ g_{|\Reg \Omega_0}$ denotes the induced metric on $\Reg \Omega_0$, and $d\eta$ and $d\vartheta$ are the volume elements of $S^{n-1}$ and $N^1_\xi \Ocal_\xi=\{v \in N_\xi \Ocal_\xi: \norm{v}=1\} $, respectively. Note that \eqref{eq:fin} implies that 
\bqn
\int_{N_\xi\Ocal_\xi} \1_{ \widetilde{\mathcal{RV}}_{c,\mu}}(x,\xi) \frac{d\sigma_{N_\xi\Ocal_\xi}(x)}{\vol \Ocal_{(x,\xi)}}
\eqn
is $\L^1$-integrable on $\rn_{(H_0)}$ as a function of $\xi$. Now, the condition $b(x,\xi) <  c(h^{\delta-\omega}+d)(x,\xi/\mu)$ implies that $|\xi| < c_1$, see Part I, equation (60); here, and in what follows, $c_i>0$ will denote positive constants. Hence,
\begin{align*}
 \widetilde{\mathcal{RV}}_{c,\mu}^{(2)}\subset &\mklm{(x,\xi) \in {\bf{X}} \times \R^n: c_0 \mu^{\epsilon_2} \leq |\xi| \leq c_1 ,\, \dist(x, \gd {\bf{X}}) < c_2 |\xi|^{-\delta} \mu^\delta}\\
  \cup&\mklm{ (x,\xi) \in {\bf{X}}\times \R^n: |\xi| < c_0\mu^{\epsilon_2} }\\
  \subset &[ (\gd {\bf{X}})_{c_3 \mu^{\delta (1-\epsilon_2)}} \times B^n(c_1)] \cup [{\bf{X}}\times B^n(c_0 \mu^{\epsilon_2})],
\end{align*}
where $B^n(\rho)$ denotes the ball of radius $\rho$ in n-dimensional Euclidean space, and $1>\epsilon_2>0$ will be chosen later. On the other hand,  the proof of Lemma 18 in Part I implies that, for small $\mu$, and some $0 < \epsilon_1<1$ to be specified later,
\begin{align*}
 \widetilde{\mathcal{RV}}_{c,\mu}^{(1)}\subset &\mklm{(x,\xi) \in {\bf{X}} \times \R^n: c_4 \leq |\xi| \leq c_1 ,\,  \big |  1-1/a_{2m}(x,\xi) \big |  \leq c_5 \mu^{\delta-\omega}}\cup [ {\bf{X}} \times B^n(\mu^{\epsilon_1})].
\end{align*}
Now, using the parametrization of $\Reg \Omega_0$ specified above, one sees that for small $\rho>0$ 
\begin{gather*}
\int_{ \Reg \Omega_{0} \cap[ {\bf{X}} \times B^n(\rho)]}  \frac{d\sigma_{\Reg \Omega_0}(z)}{\vol \Ocal_z}=
\int_{0}^{\rho }   \int_{S^{n-1}_{(H_0)}} \Big (
  \int_{N_{s\eta} \Ocal_{s\eta}}  \1_{\bf X}(x)\frac{ d\sigma_{N_{s\eta}\Ocal_{s\eta}}(x)}{\vol \Ocal_{(x,s\eta)}} \Big )   s^{n-1}  \,ds \,d\eta= O(\rho),
\end{gather*}
where we took into account took that $\vol \Ocal_{(x,s\eta)}$ is at most of order $s^\kappa$ for small $s$, and $\kappa\leq n-1$. 
Therefore, the restriction of the integral \eqref{eq:21a}  to $\Reg \Omega_{0} \cap
 \widetilde{\mathcal{RV}}_{c,\mu}^{(1)}$ can be estimated from above by 
\begin{gather*}
\int_{ \Reg \Omega_{0}}  \1_{\mklm{(x,\xi) \in {\bf{X}} \times \R^n: c_4 \leq |\xi| \leq c_1 ,\,  \big |  1-1/a_{2m}(x,\xi) \big |  \leq c_5 \mu^{\delta-\omega}}}(z) \frac{d\sigma_{\Reg \Omega_0}(z)}{\vol \Ocal_z}
 + O(\mu^{\epsilon_1}).
 \end{gather*}
 Now, by letting $x \in \rn_{(H_0)}$, $\xi \in N_x \Ocal_x$, and interchanging the roles of $x$ and $\xi$, we obtain  
 \begin{gather*}
\int_{{\bf{X}}\cap \rn_{(H_0)}} \Big (
 \int_{c_4}^{c_1}   \int_{N^1_{x} \Ocal_{x}}  \1_{\mklm{(x',\xi):   \big |  1-1/a_{2m}(x',\xi) \big |  \leq c_5 \mu^{\delta-\omega}}}(x,s\eta) \frac{s^{n-\kappa-1}  \,ds \,d\eta }{\vol \Ocal_{(x,s\eta)}} \Big ) \,dx \\
= \int_{{\bf{X}}\cap \rn_{(H_0)}} \Big (
 \int _{\mklm{\varsigma: |\varsigma-1| \leq c_5 \mu^{\delta-\omega}}}
  \int_{N^1_{x} \Ocal_{x}}  \varsigma^{-1} \Big (  \frac 1 {\varsigma a_{2m}(x,\eta)}\Big )^{\frac{n-\kappa}{2m}}
     \frac{ \1_{[c_1,c_4]} ( (\varsigma a_{2m}(x,\eta))^{-\frac 1{2m}}) d\varsigma \,d\eta   }{\vol \Ocal_{(x, (\varsigma a_{2m}(x,\eta))^{-1/2m} \eta)}} \Big ) \,dx\\
   \leq c_6  \int_{\mklm{\varsigma: |\varsigma-1| \leq c_5 \mu^{\delta-\omega}}} d\varsigma=O(\mu^{\delta-\omega}),
 \end{gather*}
where  we made the change of variables $\varsigma= |\xi|^{-2m} / a_{2m}(x, \xi/ |\xi|)=s^{-2m}/ a_{2m}(x,\eta)$, and used the fact that $(1+z)^\beta-(1-z)^\beta=O(|z|)$ for arbitrary $z \in \C$, $|z| <1$, and $\beta \in \R$. Note that due to the ellipticity condition \eqref{ellip1}, $a_{2m}(x,\eta)$ is positive for $x\in {\bf{X}}$.
Putting  $\epsilon_1=\delta-\omega=2\delta -1/2$ therefore yields 
\bqn
\int_{ \Reg \Omega_{0} \cap\widetilde{\mathcal{RV}}^{(1)}_{c,\mu}}  \frac{d\sigma_{\Reg \Omega_0}(z)}{\vol \Ocal_z} =O(\mu^{2\delta-\frac 1 2}).
\eqn
Similarly, for small $\mu$, the restriction of the integral \eqref{eq:21a} to $\Reg \Omega_{0}\cap  \widetilde{\mathcal{RV}}_{c,\mu}^{(2)}$ can be estimated from above by 
\begin{align*}
\int_{ \Reg \Omega_{0}}  &\1_{[ (\gd {\bf{X}})_{c_3 \mu^{\delta (1-\epsilon_2)}} \times B^n(c_1)] }(z) \frac{d\sigma_{\Reg \Omega_0}(z)}{\vol \Ocal_z}
 + O(\mu^{\epsilon_2})\\
&=\int_{\rn_{(H_0)} \cap(\gd {\bf{X}})_{c_3 \mu^{\delta (1-\epsilon_2)}}} \Big ( \int_{N_x \Ocal_x \cap B^n(c_1)} \frac{d\sigma_{N_x\Ocal_x}(\xi)}{\vol \Ocal_{(x,\xi)}} \Big ) dx +O(\mu^{\eps_2})\\
&\leq c_7\,  \vol  (\gd {\bf{X}})_{c_3 \mu^{\delta (1-\epsilon_2)}} +O(\mu^{\eps_2})
= O(\mu^{\delta(1-\epsilon_2)}) +O(\mu^{\epsilon_2})=O(\mu^{\frac\delta{1+\delta}})
 \end{align*}
by Assumption \ref{assump2}, where we put $\epsilon_2 =\delta/(1+\delta)$, and  took into account that, since $0 \not\in \gd {\bf{X}}$, the integrand of the last integral over $x$ is bounded on $\rn_{(H_0)}\cap (\gd {\bf{X}})_{c_3 \mu^{\delta (1-\epsilon_2)}}$ by some constant independent of $\mu$.  The assertion of the lemma now follows.
\end{proof}
Now, for $x \in {\bf{X}}$, $|\xi| > \mu$, the condition $b_{\mu^{-2m}} (x,\xi) < 0$ is equivalent to $a_{2m}(x,\xi) <1$, due to the ellipticity condition \eqref{ellip1}. By using arguments similar to those given in the  proof of the previous lemma one therefore computes
\begin{align}
\label{eq:37d}
\begin{split}
\int_{\Reg \Omega_0\cap\widetilde{W}_\mu} \frac {d\sigma_{\Reg \Omega_0}(z)}{\vol \mathcal{O}_{z}} &\leq
\int_{\Reg \Omega_0\cap [{\bf{X}} \times B^n(\mu)]} \frac {d\sigma_{\Reg \Omega_0}(z)}{\vol \mathcal{O}_{z}}+\int_{\Reg \Omega_0} \1_{(-\infty, 1]}(a_{2m}(z)) \frac {d\sigma_{\Reg \Omega_0}(z)}{\vol \mathcal{O}_{z}}\\
&= O(\mu) +\int_{\Reg \Omega_{0}/G} \1_{(-\infty, 1]}(a_{2m}([z]) d\sigma_{\Reg \Omega_0/G}([z]) \\&=O(\mu)+\vol([a_{2m}^{-1}((-\infty, 1])\cap \Reg \Omega_{0}]/G),
\end{split}
\end{align}
where we took into account Equation (3.37) in \cite{cassanas}. Here the latter volume is defined in the sense of  \cite{gallot}, Section 3.H.2.
This finishes the computation of the leading term. Collecting everything together, we obtain
\begin{proposition}\label{prop:5} 
As $\lambda \to +\infty$, one has
 \bqn
 \Big | \tr P_\chi \E_\lambda -   \frac {d_\chi [\rho_{\chi|H_0}:1]}{(2\pi)^{n-\kappa}}  \vol([a_{2m}^{-1}((-\infty, 1])\cap \Omega_{0}]/G)  \,    \lambda^{(n-\kappa)/2m} \Big | = O(\lambda^{(n-\kappa-1/4)/2m}),
 \eqn
 Furthermore, a similar result  holds for the trace of $(P_\chi \E_\lambda)^2$, too.
 \end{proposition}
 \begin{proof}
 Since $\tr P_\chi \E_\lambda= d_\chi \lambda^{n/2m} I(\lambda^{-1/2m})+O(1)$, the assertion follows with Theorem \ref{asymp} and Proposition \ref{computhess}, together with Equations \eqref{eq:37c}, \eqref{eq:37d}, and Lemma \ref{lemma:10}, by taking into account that
\bqn
\max_{\delta \in (1/4,1/2)} \min \big(\frac \delta{1+\delta}, 1-2\delta, 2\delta -\frac 1 2\big )=\frac 1 4.
\eqn  Finally, if  in all the previous computations  $\E_\lambda$ is replaced by $\E_\lambda^2$, we obtain a similar estimate for the trace of $P_\chi \E_\lambda\cdot P_\chi \E_\lambda=P_\chi \E_\lambda^2$.
 \end{proof}

\section{Proof of the main result}

As a consequence of Lemma 11 of Part I, and Proposition \ref{prop:5}, we get the following
\begin{theorem}
\label{thm:2}
Let $N^{\E_\lambda}_\chi$ be the number of eigenvalues of $\E_\lambda$ which are $\geq 1/2$ and whose eigenfunctions are contained in the $\chi$-isotypic component $\H_\chi$ of $\L^2(\R^n)$, and assume that Assumptions \ref{assump1} and \ref{assump2} are satisfied. Then
\bqn
 \Big | N^{\E_\lambda}_\chi -   \frac {d_\chi [\rho_{\chi|H_0}:1]}{(2\pi)^{n-\kappa}}  \vol([a_{2m}^{-1}((-\infty, 1])\cap \Omega_{0}]/G)  \,    \lambda^{(n-\kappa)/2m} \Big | = O(\lambda^{(n-\kappa-1/4)/2m}),
\eqn
as $\lambda \to +\infty$.
\qed
\end{theorem}

Similar estimates for  the traces of $\tilde \F_\lambda$ and $\F_\lambda$ can be derived as well, and using Lemma 12 of Part I we obtain
\begin{theorem}
\label{thm:2a}
Let $M^{\F_\lambda}_\chi$ be the number of eigenvalues of $\F_\lambda$ which are $\geq 1/2$ and whose eigenfunctions are contained in the $\chi$-isotypic component $\H_\chi$ of $\L^2(\R^n)$. Then under Assumptions \ref{assump1} and \ref{assump2} one has
\bqn
 \Big | M^{\F_\lambda}_\chi -   \frac {d_\chi [\rho_{\chi|H_0}:1]}{(2\pi)^{n-\kappa}}  \vol([a_{2m}^{-1}((-\infty, 1])\cap \Omega_{0}]/G)  \,    \lambda^{(n-\kappa)/2m} \Big | = O(\lambda^{(n-\kappa-1/4)/2m}),
\eqn
as $\lambda \to +\infty$.
\end{theorem}
\begin{proof}
The proof   is similar to the one of Theorem \ref{thm:2}; in analogy to Equation \eqref{eq:23b} one has
\bqn
\sigma^l(\F_\lambda) = (\eta_{2}^2 \chi^+_\lambda)^2 ( 3- 2 \eta_{2}^2\chi^+_\lambda) + f_\lambda + r_\lambda,
\eqn
where $r_\lambda \in S^{-\infty}(h^{-2\delta}g,1)$, and $ f_\lambda \in S(h^{-2\delta} g, h^{1 - 2\delta})$,  everything uniformly in $\lambda$. Again we have $\supp f_\lambda \subset \mathcal{RV}_{c,\lambda}$ for sufficiently large $c$, and  $\int \int |r_\lambda(x,\xi)| \, dx\, \dbar \xi \leq C$ for some constant $C>0$ independent of $\lambda$, so that in order  to study the asymptotic behavior of $\tr P_\chi \F_\lambda$, we can  restrict ourselves to the integral
 \bqn
\int_G \int \int  \overline{\chi(k)} e^{i(x -kx) \xi} ((\eta_{2}^2 \chi^+_\lambda )^2( 3- 2 \eta_{2}^2\chi^+_\lambda)+f_\lambda) (x ,\xi) dx \, \dbar \xi \, dk .
 \eqn
 An application of  the method of the stationary phase then yields the desired result.
\end{proof}
We are now in position to prove our main result. In the case  $G=\{\1\}$, one has $\Omega_0=\R^{2n}$, and we simply obtain Theorem 13.1 of \cite{levendorskii}.
\begin{theorem}
\label{thm:1}
Let $G$ be a compact  group of isometries in Euclidean space $\R^n$, $H_0$ a principal isotropy group, and ${\bf{X}} \subset \R^n$ a bounded open set invariant under $G$. Assume that
\begin{enumerate}
\item[(i)]  for sufficiently small $\rho >0$, $\vol ( \gd {\bf{X}} )_\rho \leq c \rho$, where $c>0$ is a constant independent of $\rho$, and $0\not\in \gd {\bf{X}}$;
\item[(ii)]  the set $\Sing \rn=\rn\setminus \rn_{(H_0)}$ is included in a strict vector subspace $F$ of $\rn$ of dimension $r<n$.
\end{enumerate}
Let  further $A_0$ be a  symmetric, classical pseudodifferential operator in $\L^2(\R^n)$ of order $2m$ with principal symbol $a_{2m}$ that commutes with the regular representation $T$ of $G$, and assume that $A_0$ satisfies the ellipticity condition \eqref{ellip1}. Consider further  the Friedrichs extension of the operator
\bqn
\mathrm{res} \circ A_0 \circ \mathrm{ext}: \CT({\bf{X}}) \longrightarrow \L^2({\bf{X}}),
\eqn
and denote it by $A$. Then $A$ has discrete spectrum. Furthermore, if  $N_\chi(\lambda)$ denotes the number of  eigenvalues of $A$ less or equal $\lambda$ and  with eigenfunctions in the  $\chi$-isotypic component $\mathrm{res}\, \H_\chi$ of $\L^2({\bf{X}})$, and $\kappa =\dim H_0$, then
\begin{equation*}
N_\chi(\lambda)=\frac {d_\chi [\rho_{\chi|H_0}:1]}{(2\pi)^{n-\kappa}}  \vol([ a^{-1}_{2m}( (-\infty,1])\cap \Omega_{0}]/G) \,    \lambda^{(n-\kappa)/2m} +O(\lambda^{(n-\kappa-1/4)/2m}),
\end{equation*}
 where $d_\chi$ denotes the dimension of  any unitary irreducible representation $\rho_\chi$ determined by the character $\chi$, and    $[\rho_{\chi|H_0}:1]$ is the multiplicity of the trivial representation in the restriction of $\rho_\chi$ to  $H_0$.  
\end{theorem}
\begin{proof}
The disreteness of the spectrum was already shown in Proposition \ref{cpctres}. Now, by  Theorems 5 and 6 of Part I, there exist constants $C_i>0$ independent of $\lambda$ such that
$$N^{\E_\lambda}_\chi -C_1 \leq  \mathcal{N}(A_0-\lambda \1, \H_\chi\cap \CT({\bf{X}})) \leq M_\chi^{\F_\lambda} +C_2.$$
 Theorems \ref{thm:2} and \ref{thm:2a} then yield the estimate
\bqn
 \Big | N_\chi(\lambda)  -   \frac {d_\chi [\rho_{\chi|H_0}:1]}{(2\pi)^{n-\kappa}}  \vol([a_{2m}^{-1}((-\infty, 1])\cap \Omega_{0}]/G)  \,    \lambda^{(n-\kappa)/2m} \Big | = O(\lambda^{(n-\kappa-1/4)/2m}).
\eqn
The proof of the theorem is now complete.
\end{proof}


\providecommand{\bysame}{\leavevmode\hbox to3em{\hrulefill}\thinspace}
\providecommand{\MR}{\relax\ifhmode\unskip\space\fi MR }
\providecommand{\MRhref}[2]{%
  \href{http://www.ams.org/mathscinet-getitem?mr=#1}{#2}
}
\providecommand{\href}[2]{#2}


\begin{thebibliography}{10}

\bibitem{bredon}
G.E. Bredon, \emph{Introduction to compact transformation groups}, Academic
  Press, New York, 1972, Pure and Applied Mathematics, Vol. 46.

\bibitem{ivrii-bronstein03}
M.~Bronstein and V.~Ivrii, \emph{{Sharp spectral asymptotics for operators with
  irregular coefficients I. Pushing the limits}}, Comm. Partial Diff. Equations
  \textbf{28} (2003), 83--102.

\bibitem{bruening-heintze}
J.~Br\"uning and E.~Heintze, \emph{Representations of compact {Lie groups} and
  elliptic operators}, Inventiones math. \textbf{50} (1979), 169--203.

\bibitem{cassanas}
R.~Cassanas, \emph{Reduced {G}utzwiller formula with symmetry: {C}ase of a lie
  group}, J. Math. Pures Appl. \textbf{85} (2006), 719--742.

\bibitem{combescure-ralston-robert}
M.~Combescure, J.~Ralston, and D.~Robert, \emph{A proof of the {G}utzwiller
  semiclassical trace formula using coherent states decomposition}, Comm. Math.
  Phys. \textbf{202} (1999), 463--480.

\bibitem{donnelly}
H.~Donnelly, \emph{{G-spaces, the asymptotic splitting of $L^2(M)$ into
  irreducibles}}, Math. Ann. \textbf{237} (1978), 23--40.

\bibitem{helffer-elhouakmi}
Z.~El~Houakmi and B.~Helffer, \emph{Comportement semi-classique en pr\'esence
  de sym\'etries: action d'un groupe de {L}ie compact}, Asymptotic Anal.
  \textbf{5} (1991), no.~2, 91--113.

\bibitem{emmrich-roemer}
C.~Emmrich and H.~R{\"o}mer, \emph{Orbifolds as configuration spaces of systems
  with gauge symmetries}, Comm. Math. Phys. \textbf{129} (1990), no.~1, 69--94.

\bibitem{feigin}
V.~I. Feigin, \emph{Asymptotic distribution of eigenvalues for hypoelliptic
  systems in {$R^{n}$}}, Math. USSR Sbornik \textbf{28} (1976), no.~4,
  533--552.

\bibitem{gallot}
S.~Gallot, D.~Hulin, and J.~Lafontaine, \emph{Riemannian geometry},
  Springer-Verlag, third edition, 2004.

\bibitem{grigis-sjoestrand}
A.~Grigis and J.~Sj{\"o}strand, \emph{Microlocal analysis for differential
  operators}, London Mathematical Society Lecture Note Series, vol. 196,
  Cambridge University Press, 1994.

\bibitem{helffer-robert86}
B.~Helffer and D.~Robert, \emph{Etude du spectre pour un op\'eratour
  globalement elliptique dont le symbole de {Weyl} pr\'esente des sym\'etries
  {II}}, Amer. J. Math. \textbf{108} (1986), 973--1000.

\bibitem{hoermander68}
L.~H\"ormander, \emph{The spectral function of an elliptic operator}, Acta
  Math. \textbf{121} (1968), 193--218.

\bibitem{hoermander79}
L.~H\"{o}rmander, \emph{The {W}eyl calculus of pseudo-differential operators},
  Comm. Pure Appl. Math. \textbf{32} (1979), 359--443.

\bibitem{hoermanderI}
L.~H{\"{o}}rmander, \emph{The analysis of linear partial differential
  operators}, vol.~I, Springer--Verlag, Berlin, Heidelberg, New York, 1983.

\bibitem{ivrii03}
V.~Ivrii, \emph{{Sharp spectral asymptotics for operators with irregular
  coefficients II. Domains with boundaries and degenerations}}, Comm. Partial
  Diff. Equations \textbf{28} (2003), 103--128.

\bibitem{kawakubo}
K.~Kawakubo, \emph{The theory of transformation groups}, The Clarendon Press
  Oxford University Press, New York, 1991.

\bibitem{levendorskii}
S.~Z. Levendorskii, \emph{Asymptotic distribution of eigenvalues}, Kluwer
  Academic Publishers, Dordrecht, Boston, London, 1990.

\bibitem{montgomery-samelson}
D.~Montgomery and H.~Samelson, \emph{Transformation groups of spheres}, Ann. of
  Math. (1943), no.~44, 457--470.

\bibitem{ortega-ratiu}
J.P. Ortega and T.S. Ratiu, \emph{Momentum maps and {H}amiltonian reduction},
  Progress in Mathematics, vol. 222, Birkh\"auser Boston Inc., Boston, MA,
  2004.

\bibitem{ramacher07}
P.~Ramacher, \emph{Reduced {Weyl} asymptotics for pseudodifferential operators
  on bounded domains {I}. {T}he finite group case}, 2007.

\bibitem{reed-simonII}
M.~Reed and B.~Simon, \emph{Methods of modern mathematical physics. {II}.
  {F}ourier analysis, self-adjointness}, Academic Press, New York, 1975.

\bibitem{shubin}
M.~A. Shubin, \emph{Pseudodifferential operators and spectral theory}, 2nd
  edition, Springer--Verlag, Berlin, Heidelberg, New York, 2001.

\bibitem{tulovsky-shubin}
V.~N. Tulovsky and M.~A. Shubin, \emph{On the asymptotic distribution of
  eigenvalues of pseudodifferential operators in {$\R^n$}}, Math. Trans.
  \textbf{92} (1973), no.~4, 571--588.

\bibitem{weyl}
H.~Weyl, \emph{{Das asymptotische Verteilungsgesetz der Eigenwerte linearer
  partieller Differentialgleichungen (mit einer Anwendung auf die Theorie der
  Hohlraumstrahlung)}}, Math. Ann. \textbf{71} (1912), 441--479.

\end{thebibliography}


\end{document}